\documentclass[a4paper]{article}
\pdfoutput=1
\usepackage[english]{babel}
\usepackage{amsmath}
\usepackage{amsthm}
\usepackage{amssymb}
\usepackage{mathtools}
\usepackage{enumitem}
\usepackage{hyperref}
\usepackage{color}
\newtheorem{theo}{Theorem}
\newtheorem{defi}{Definition}
\newtheorem{lemma}{Lemma}
\newtheorem{claim}{Claim}

\newtheorem{cor}[theo]{Corollary}
\newtheorem{prob}{Problem}

\newtheorem{prop}{Proposition}

\newcommand{\N}{\mathbb{N}}

\title{Ramsey upper density of infinite graphs}
\date{\vspace{-5ex}}
\author{Ander Lamaison\thanks{Funded by the Deutsche Forschungsgemeinschaft (DFG, German Research Foundation) under Germany's Excellence Strategy - The Berlin Mathematics Research Center MATH+ (EXC-2046/1, project ID: 390685689).}\\ Freie Universit\"at Berlin}

\begin{document}
\maketitle

\begin{abstract}

For a fixed infinite graph $H$, we study the largest density of a monochromatic subgraph isomorphic to $H$ that can be found in every two-coloring of the edges of $K_\N$. This is called the Ramsey upper density of $H$, and was introduced by Erd\H{o}s and Galvin. Recently \cite{CDLL}, the Ramsey upper density of the infinite path was determined. Here, we find the value of this density for all locally finite graphs $H$ up to a factor of 2, answering a question of DeBiasio and McKenney. 

We also find the exact density for a wide class of bipartite graphs, including all locally finite forests. Our approach relates this problem to the solution of an optimization problem for continuous functions. We show that, under certain conditions, the density depends only on the chromatic number of $H$, the number of components of $H$, and the expansion ratio $|N(I)|/|I|$ of the independent sets of $H$.

\end{abstract}

\section{Introduction}\label{sec:intro}

Let $K_\N$ be the complete graph on the natural numbers. Let $H$ be a countably infinite graph (meaning that the vertex set has the same cardinality as $\N$). Suppose that the edges of $K_\N$ are colored red or blue. We can find a monochromatic subgraph $H'\subseteq K_\N$ isomorphic to $H$. For example, using Ramsey's Theorem, we can produce $H'$ by finding a bijection between $V(H)$ and the vertices of a monochromatic infinite clique. Out of all possible subgraphs $H'$, we want to find one which maximizes its density. To measure the density, we use the following definition:

\begin{defi} Let $S\subseteq \N$. We define the \textit{upper density} of $S$ (in this paper shortened to \textit{density}) as \[\bar d(S)=\limsup\limits_{n\rightarrow\infty}\frac{|S\cap[n]|}{n}.\] If $H'\subseteq K_\N$, we define $\bar d(H')=\bar d(V(H'))$.\end{defi}

We are interested in an extremal question: if $H$ is a fixed graph, what is the maximum density of $H'$ that we can find in every red-blue coloring of $K_\N$? We call this value the Ramsey upper density of $H$.

\begin{defi} Let $H$ be a countably infinite graph. We define its Ramsey upper density $\rho(H)$ as the supremum of the values of $\lambda$ for which, for every two-coloring of $E(K_\N)$, there exists a monochromatic subgraph $H'\subseteq K_\N$, isomorphic to $H$, with $\bar d(H')\geq \lambda$.
\end{defi}

The study of this parameter was initiated by Erd\H{o}s and Galvin \cite{ErdGal} for the particular case $H=P_\infty$, the one-way infinite path. They proved that $2/3\leq\rho(P_\infty)\leq8/9$. After some improvements on these bounds in \cite{BiaKen, LoSanWan}, the exact value of $P_\infty$ was determined by Corsten, DeBiasio, Lang and the author \cite{CDLL} as $\rho(P_\infty)=(12+\sqrt{8})/17\approx 0.87226$. The parameter $\rho(H)$ for general $H$ was first introduced by DeBiasio and McKenney~\cite{BiaKen}.

Our aim in this paper is to give bounds on $\rho(H)$ for a wider family of graphs $H$. These results can be found futher down in the introduction, although some of the more general bounds, with a more involved statement, are left for later. As it will turn out, three parameters play an important role in the value of $\rho(H)$: its chromatic number, the number of components and the expansion properties of its independent sets.

\subsection{Notation}

An infinite graph is \textit{locally finite} if every vertex has finite degree. The bounds that we will show in this paper apply only to locally finite graphs $H$.

Given $S\subseteq V(H)$, we will denote by $N(S)=\left(\cup_{v\in S} N(v)\right)\setminus S$ the set of vertices outside $S$ with a neighbor in $S$. We let $\mu(H,n)$ be the minimum value of $|N(I)|$, where $I$ is an independent set in $H$ of size $n$. We say that a set $I$ is \textit{doubly independent} if both $I$ and $N(I)$ are independent.

We say that a family $\{S_1, S_2, \dots\}$ of subsets of $V(H)$ is \textit{concentrated} in at most $k$ components if there are components $C_1, C_2, \dots, C_s$ of $H$, with $s\leq k$, such that all but finitely many sets $S_i$ intersect some component $C_j$. We say that $V(H)$ is concentrated in at most $k$ components if $\{\{v\}:v\in V(H)\}$ is concentrated in at most $k$ components.

On some occasions we will use $C\in \{R,B\}$ to designate a color (red or blue). When this happens, $\bar C$ will denote the other color. In a graph $G$ with colored vertices, we will use $C_G$ to refer to the set of vertices of color $C$. If there is no ambiguity, we will omit the subindex $G$.

If $F$ is a finite graph, we denote by $\omega\cdot F$ the graph obtained by taking the disjoint union of a countably infinite number of copies of $F$.

Finally, we define a function $f(x)$ which will be crucial in relating the values of $\rho(H)$ and $|N(I)|/|I|$, where $I$ is an independent or a doubly independent set of $H$. Unfortunately, there is no satisfying intuition for why this particular choice of $f(x)$, and not another, is behind the relation between these two parameters. It is interesting however that the same function $f(x)$ arises from the study of upper bounds and lower bounds for $\rho(H)$. 

Since the definition of $f(x)$ is quite complicated and its comprehension is not essential to the appreciation of our results, we encourage the reader to skip it for now. For the reading of the introduction, knowing that such a function exists is enough. Of course, for the reading of the proofs, the precise definition becomes necessary.

\begin{defi}\label{defif}
Let $\gamma\in(-1,1)$. For a continuous function $g(x):[0,+\infty)\rightarrow\mathbb{R}$, define \[\Gamma^+_\gamma(g,t)=\min\{x:\gamma x+g(x)\geq t\}\hskip 1cm \Gamma^-_\gamma(g,t)=\min\{x:\gamma x-g(x)\geq t\},\]where we take the minimum of the empty set to be $+\infty$. We define $h(\gamma)$ to be the infimum, over all 1-Lipschitz\footnote{A 1-Lipschitz function is a function satisfying $|f(x)-f(y)|\leq |x-y|$ for every $x,y$ in the domain.} functions $g$ with $g(0)=0$, of \begin{equation}\label{optiprob}h(\gamma)=\inf\limits_{g}\limsup\limits_{t\rightarrow\infty}\frac{\Gamma^+_\gamma(g,t)+\Gamma^-_\gamma(g,t)}{t}.\end{equation}

Define $f:(0,+\infty)\rightarrow\mathbb{R}$ as \[f(\lambda)=1-\frac{1}{\frac{2\lambda}{(1+\lambda)^2}h\left(\frac{\lambda-1}{\lambda+1}\right)+\frac{2\lambda}{1+\lambda}}.\] We define $f(0)=1$ and $f(+\infty)=1/2$ (by \eqref{func} below, we have $\lim\limits_{t\rightarrow 0}f(t)=1$ and $\lim\limits_{t\rightarrow+\infty}f(t)=1/2$.)
\end{defi}

In Appendix \ref{appf} we prove some properties of $f(x)$, including the following bounds:

\begin{equation}\label{func}\frac{x+1}{2x+1}\leq f(x) \leq \left\{
  \begin{array}{ccc}
    \frac{2x^2+3x+7+2\sqrt{x+1}}{4x^2+4x+9} & \text{for} & 0\leq x< 3,\\
    & & \\
    \frac{x+1}{2x} & \text{for} & x \geq 3.
  \end{array}
\right.\end{equation}
The upper bound is tight for $x\in[0,1]$, and we conjecture\footnote{An extended abstract for this paper, published in {\it Acta Math. Univ. Comenianae} for EUROCOMB 2019, stated this as proved. Since then, a mistake in the proof has been found.} that it is tight everywhere. Observe that $f(1)=(12+\sqrt{8})/17=\rho(P_\infty)$.

\begin{figure}
\begin{centering}
\includegraphics[width=60mm]{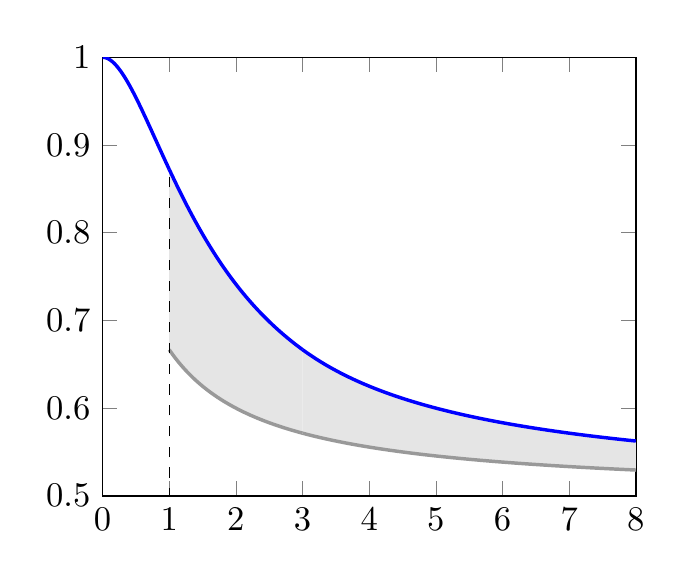}
\caption{Plot of the function $f(x)$ on the interval $[0,1]$, and the upper and lower bounds elsewhere. The conjectured value is given in blue.}
\label{nbex}
\par\end{centering}
\end{figure}

\subsection{Results}

We will now give a few bounds on $\rho(H)$, some of which apply for all locally finite graphs and some of which apply only for particular families. In many cases the specific results follow from other results which are more general, but which have more involved statements. These will be stated in later sections.

For locally finite graphs, knowing the chromatic number and the number of components is enough to determine $\rho(H)$ up to a factor of 2.

\begin{theo}\label{compmt}Let $H$ be a locally finite graph. \begin{enumerate}[label=(\roman*)]\item
\label{compi} If $H$ has infinitely many components, then $\rho(H)\geq 1/2$. 
\item If $H$ has finitely many components: 
	\begin{enumerate}\item\label{compii} If $H$ has infinite chromatic number, then $\rho(H)=0$. 
\item\label{compiii} If $H$ has finite chromatic number, then \[\min\left\{\frac{b}{2(\chi(H)-1)}, \frac12\right\}\leq\rho(H)\leq\min\left\{\frac{b}{\chi(H)-1},1\right\},\]where $b$ is the number of infinite components of $H$.\end{enumerate}\end{enumerate}\end{theo}

This theorem answers a question in \cite{BiaKen}, which asks whether for every $\Delta$ there exists a constant $c>0$ such that every graph with maximum degree at most $\Delta$ has Ramsey upper density at least $c$:

\begin{cor}\label{bkquestion}If $H$ has maximum degree at most $\Delta$, then $\rho(H)\geq 1/(2\Delta)$.\end{cor}

Let $P_\infty^k$ be the $k$-th power of the infinite path, that is, the graph on $\N$ in which $x$ and $y$ are connected if $|x-y|\leq k$. Elekes et al. \cite{ESSS} showed that, in every two-coloring of $K_\N$, the vertex set can be partitioned into at most $2^{2k-1}$ monocromatic copies of $P_\infty^k$ plus a finite set, and the number of copies can be reduced to four for $P_\infty^2$. DeBiasio and McKenney \cite{BiaKen} pointed out that this implies $\rho(P_\infty^2)\geq 1/4$ and $\rho(P_\infty^k)\geq 2^{1-2k}$. Theorem \ref{compmt} improves the bound for $k\geq 3$ to $\rho(P_\infty^k)\geq 1/(2k)$.

While no graph $H$ is known for which the lower bound in \ref{compiii} is tight and not equal to $1/2$, the upper bound is tight in the following example. Let $T$ be the tree formed by an infinite path $v_1v_2v_3\dots$, in which we attach $i$ leaves to $v_i$ for every $i\in\N$. Then $\rho(b\cdot T+K_a)=b/(a-1)$ for every $1\leq b<a$, where $b\cdot T+K_a$ denotes the disjoint union of $b$ copies of $T$ and an $a$-clique. The lower bound will follow from Theorem \ref{mainlb}.

Another upper bound that applies to all locally finite graphs is related to the expansion of its independent sets:

\begin{theo}\label{mainub} Let $H$ be a locally finite graph. Then \[\rho(H)\leq f\left(\liminf\limits_{n\rightarrow\infty}\frac{\mu(H,n)}n\right)\].
\end{theo}

There are many graphs for which the bound in Theorem \ref{mainub} is tight. The following theorem captures some of them.

\begin{theo}\label{forest} Let $H$ be a locally finite forest, or a locally finite bipartite graph in which every orbit of the automorphism group acting on $V(H)$ has infinite size. Then \[\rho(H)=f\left(\liminf\limits_{n\rightarrow\infty}\frac{\mu(H,n)}{n}\right)\]. 
\end{theo}

This is a particular case of a more general condition on bipartite graphs that is sufficient for Theorem \ref{mainub} to be tight. That condition is stated later as Theorem \ref{mainthmbip}. The following corollaries illustrate some examples of graphs for which Theorem \ref{forest} applies:

\begin{cor}\label{kary} Let $T_k$ be the infinite $k$-ary tree, that is, the rooted tree in which every vertex has $k$ children. Then $\rho(T_k)=f(k)$.\end{cor}

\begin{cor}\label{grid} Let $\mathrm{Grid}_d$ be the infinite $d$-dimensional grid, that is, the graph on $\mathbb{Z}^d$ where two vertices are connected if they are at Euclidean distance 1. Then $\rho(\mathrm{Grid}_d)=f(1)=(12+\sqrt{8})/17\approx 0.87226$.\end{cor}

\begin{cor}\label{bipfac} Let $F$ be a finite bipartite graph. Then \[\rho(\omega\cdot F)=f\left(\min\limits_{\substack{I\text{ indep. in }F\\I\neq\emptyset}}\frac{|N(I)|}{|I|}\right).\] In particular, for every $1\leq a\leq b$ we have \[\rho(\omega\cdot K_{a,b})=f\left(\frac ab\right)=\frac{2\left(\frac ab\right)^2+3\left(\frac ab\right)+7+2\sqrt{\frac ab+1}}{4\left(\frac ab\right)^2+4\left(\frac ab\right)+9}.\]\end{cor}

In a finite bipartite graph $F$, there is always an independent set satisfyng $|N(I)|\leq|I|$ (one of the two partition classes has this), so the value of $\rho(\omega\cdot F)$ always falls on the range in which $f(x)$ is known explicitly.



Finally, we will give two more lower bound in the particular case of infinite factors $\omega\cdot F$. The first one is analogous to Corollary \ref{bipfac}:

\begin{theo}\label{factornb} Let $F$ be a finite connected graph, and let $I\subseteq V(F)$ be a non-empty doubly independent set. Then $\rho(\omega\cdot F)\geq f\left(\frac{|N(I)|}{|I|}\right)$. \end{theo}

If the independent set $I\subseteq V(F)$ that minimizes $|N(I)|/|I|$ is doubly independent, then Theorem \ref{mainub} and Theorem \ref{factornb} together give the exact value for $\rho(\omega\cdot F)$. This is always true in bipartite graphs, giving another reason why Corollary \ref{bipfac} holds. Figure \ref{nbex} shows four non-bipartite graphs $F$ for which this holds. If the graph $F$ does not contain any non-empty doubly independent sets (such as $K_3$), the following lower bound can be used:

\begin{figure}
\begin{centering}
\includegraphics[width=85mm]{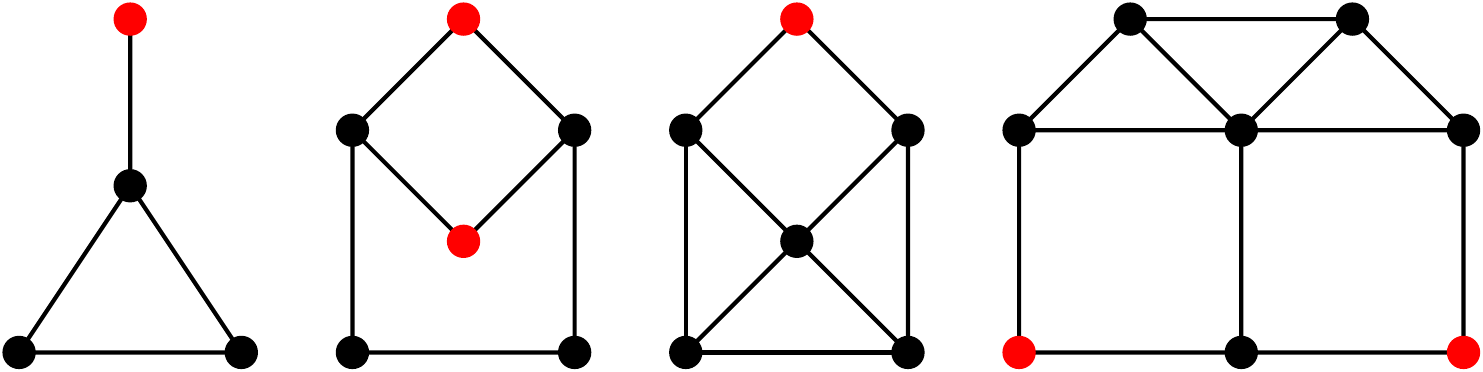}
\caption{Four non-bipartite graphs $F$ for which $\rho(\omega\cdot F)$ equals $f(1)$, $f(1)$, $f(2)$ and $f(3/2)$ respectively, with their doubly independent sets indicated.}
\label{nbex}
\par\end{centering}
\end{figure}

\begin{theo}\label{beslb} For every finite graph $F$, we have \[\rho(\omega\cdot F)\geq \frac{|V(F)|}{2|V(F)|-\alpha(F)}.\]\end{theo}

This theorem gives the best known lower bound for $\rho(\omega\cdot K_3)$. Combining Theorem \ref{beslb} with Theorem \ref{mainub} and \eqref{func} we obtain \[3/5\leq\rho(\omega\cdot K_3)\leq f(2)\leq \frac{21+\sqrt{12}}{33}\approx0.74133.\]

This paper is organized as follows: we prove the general upper bounds in Section \ref{sec:genubounds}, and the general lower bounds in Section \ref{sec:genlbounds}, besides a lemma that is left for Appendix \ref{annexlem} (the bulk of this proof is a rather long series of calculations without any interesting ideas behind). In Section \ref{sec:examples} we discuss the application of the general bounds to particular families of graphs, and obtain the remaining results above. In Section \ref{sec:openprob} we state some open questions, and in Appendix \ref{appf} we will prove some properties of $f(x)$.

\section{General upper bounds}\label{sec:genubounds}

We will prove two upper bounds in this section: the first one implies the upper bound from items \ref{compii} and \ref{compiii} from Theorem \ref{compmt}, while the second one is Theorem \ref{mainub}. In both cases we will construct a coloring of $E(K_\N)$ in which no dense monochromatic copy of $H$ exists.

\begin{theo}\label{compub} Let $H$ be a locally finite graph with chromatic number at least $a$, such that $V(H)$ is concentrated in at most $b$ components. Then $\rho(H)\leq b/(a-1)$.\end{theo}

\begin{proof} Consider the coloring of $K_\N$ in which the edge $uv$ is red iff $a-1$ divides $v-u$. The graph formed by the blue edges has chromatic number $a-1$, and thus does not contain $H$ as a subgraph. Every monochromatic copy of $H$ in this coloring is red.

The red graph consists of $a-1$ cliques $C_1, \dots, C_{a-1}$, each with $\bar d(C_i)=1/(a-1)$. The $b$ components that concentrate $V(H')$ must be each contained in a clique $C_i$. Because modifying finitely many elements does not affect the density of a set, we have $\bar d(H')\leq\bar d(C_1\cup\dots\cup C_b)=b/(a-1)$. We conclude that $\rho(H)\leq b/(a-1)$.
\end{proof}

Next we will prove Theorem \ref{mainub}. 
The intuition behind the construction to prove Theorem \ref{mainub} is as follows: suppose that we are trying to find a red copy of $H$. If we have a blue clique $K$ which has fewer than $k$ vertices neighboring $K$ through some red edge, and $\mu(H,t)=k$, then we know that fewer than $t$ vertices from $K$ can be in $H'$, because those vertices correspond to an independent set in $H$. Our goal is to find a construction that maximizes the number of vertices from $[n]$ that can be excluded from a potential red or blue $H'$ using this method.

\begin{proof}[Proof of Theorem \ref{mainub}]

Denote $\lambda=\liminf\limits_{n\rightarrow\infty}\frac{\mu(H,n)}{n}$. Let $\epsilon>0$. Let $g$ be a 1-Lipschitz function such that the upper limit in \eqref{optiprob} is less than $h(\gamma)+\epsilon$, for $\gamma=\frac{\lambda-1}{\lambda+1}$. Take an infinite set of vertices $v_1, v_2, \dots$, and arrange them from left to right in this order. Color these vertices red and blue, in such a way that among the $n$ leftmost vertices there are exactly $\lfloor(n+g(n))/2\rfloor$ red vertices (this is possible because $g$ is 1-Lipschitz). Form a two-colored complete graph by giving each edge the color of its leftmost endpoint.

There must be infinitely many vertices of each color. This is because otherwise one of the non-decreasing functions $\frac{x+g(x)}{2}$ or $\frac{x-g(x)}{2}$ is bounded, giving an absolute upper bound on $x\pm g(x)$. Then $\max\{\ell^+_\gamma(g,n), \ell^-_\gamma(g,n)\}=+\infty$ for every $n$ large enough.

Let the red vertices be $r_1, r_2, r_3, \dots$ and the blue vertices be $b_1, b_2, b_3, \dots$, according to the left-to-right order. Let $\alpha_i$ be the smallest value such that $r_{\alpha_i}$ has at most $\lambda(\alpha_i-i)$ blue vertices to its left, and $\beta_i$ the smallest value such that $b_{\beta_i}$ has at most $\lambda(\beta_i-i)$ red vertices to its left. The following discussion will not only prove the existence of $\alpha_i$ and $\beta_i$, but also give a bound on them.

Let $w=\frac{2}{1+\lambda}(\lambda i+2\lambda+2)$. Let $z^+=\Gamma^+_\gamma(g,w)$ and $z^-=\Gamma^-_\gamma(g,w)$. By continuity of $g$ and definition of $\Gamma^+_\gamma$ and $\Gamma^-_\gamma$, we have \[g(z^+)=w-\gamma z^+,\hskip 1cm g(z^-)=\gamma z^--w.\]

One can check that the following identity holds by substution of $g(z^-)$:

\[\frac{z^-+g(z^-)}{2}+2=\lambda\left(\frac{z^--g(z^-)}{2}-i-2\right).\] 

Among the $\lfloor z^-\rfloor$ leftmost vertices there are $\left\lfloor\frac{\lfloor z^-\rfloor+g(\lfloor z^-\rfloor)}{2}\right\rfloor$ red vertices and $\lfloor z^-\rfloor-\left\lfloor\frac{\lfloor z^-\rfloor+g(\lfloor z^-\rfloor)}{2}\right\rfloor$ blue vertices. Observe that

\begin{align*}\left\lfloor\frac{\lfloor z^-\rfloor+g(\lfloor z^-\rfloor)}{2}\right\rfloor\leq&\frac{\lfloor z^-\rfloor+g(\lfloor z^-\rfloor)}{2}
\leq \frac{z^-+g(z^-)}{2} <\lambda\left(\frac{z^--g(z^-)}{2}-i-2\right)\\ \leq&\lambda\left(\lfloor z^-\rfloor-\left\lfloor\frac{\lfloor z^-\rfloor+g(\lfloor z^-\rfloor)}{2}\right\rfloor-i\right).\end{align*}

If the last blue vertex among those $\lfloor z^-\rfloor$ is $b_{\tau}$, then the number of red vertices to its left is less than $\lambda(\tau-i)$, meaning that $\beta_{i}\leq\tau$. Hence

\begin{equation*}\label{betaineq}\beta_{i}\leq \lfloor z^-\rfloor-\left\lfloor\frac{\lfloor z^-\rfloor+g(\lfloor z^-\rfloor)}{2}\right\rfloor\leq \frac{z^--g(z^-)}{2}+2=\frac{1-\gamma}{2}z^-+\frac{w}{2}+2.\end{equation*}

Analogously, one has the identity \[\frac{z^+-g(z^+)}{2}+2=\lambda\left(\frac{z^++g(z^+)}{2}-i-2\right)\] and the inequality \[\lfloor z^+\rfloor-\left\lfloor\frac{\lfloor z^+\rfloor+g(\lfloor z^+\rfloor)}{2}\right\rfloor\leq\lambda\left(\left\lfloor\frac{\lfloor z^+\rfloor+g(\lfloor z^+\rfloor)}{2}\right\rfloor-i\right).\]

Hence we find \[\alpha_i\leq\left\lfloor\frac{\lfloor z^+\rfloor+g(\lfloor z^+\rfloor)}{2}\right\rfloor\leq\frac{ z^++g(z^+)}{2}=\frac{1-\gamma}{2}z^++\frac{w}{2}.\]

Adding the two values together, for $i$ large enough we have \begin{equation*}\alpha_i+\beta_i\leq\frac{1-\gamma}{2}(z^++z^-)+w+2\leq\left(\frac{1-\gamma}{2}h(\gamma)+\epsilon+1\right)\frac{2\lambda}{1+\lambda}i+o(i).\end{equation*}

Let $\phi:\mathbb{N}\rightarrow \{v_1, v_2, \dots\}$ be an arbitrary bijection satisfying $\phi([\alpha_j+\beta_j])=\{r_1, r_2, \dots, r_{\alpha_j}, b_1, b_2, \dots, b_{\beta_j}\}$ for every $j$. The function $\phi$ defines a coloring of $K_\N$, where the color of the edge $ij$ is the color of the edge $\phi(i)\phi(j)$.

Let $R$ and $B$ be the sets of positive integers $i$ whose image $\phi(i)$ is red or blue, respectively. Let $H'\subseteq K_\N$ be a monochromatic copy of $H$ in this coloring. Suppose that $H'$ is red. Let $n$ be a positive integer, and let $B_n=V(H')\cap[n]\cap B$. Because the vertices of $B_n$ form a monochromatic blue clique in our coloring of $K_\N$, the set $B_n$ must be independent in $H'$.

Let $j$ be the minimum value such that $\phi(B\cap [n])\subseteq\{b_1, b_2, \dots, b_{\beta_j}\}$. We claim first that there are at least $(1-O(\epsilon))j$ vertices in $[n]$ which do not belong to $H'$. Indeed, let $B'_n=V(H')\cap \{b_1, b_2, \dots, b_{\beta_{j-1}}\}\subseteq B_n$. From the construction of the coloring, the vertices that are connected to a vertex of $\{b_1, b_2, \dots, b_{\beta_{j-1}}\}$ through a red edge are precisely the red vertices to the left of $b_{\beta_{j-1}}$, of which there are at most $\lambda(\beta_{j-1}-(j-1))$. This means that $\mu(H,|B'_n|)\leq \lambda(\beta_{j-1}-(j-1))$. For $j$ large enough, this implies $|B'_n|\leq (1+o(1))(\beta_{j-1}-(j-1))$, and \[|[n]\setminus V(H')|\geq \beta_{j-1}-|B'_n|\geq (1+o(1))(j-1)-o(1)\beta_{j-1}.\]

Observe next that we cannot have $\beta_j=\beta_{j+1}$. This is because $b_{\beta_j-1}$, which is to the left\footnote{We cannot have $\beta_j=1$ for $j\geq 2$, because then $b_1$ would have at most $\lambda(1-j)<0$ red vertices to its left.} of $b_{\beta_j}$, has more than $\lambda((\beta_j-1)-j)=\lambda(\beta_j-(j+1))$ red vertices to its left. We thus have, by minimality of $j$, that $b_{\beta_{j+1}}\notin\phi([n])$, and by construction of $\phi$ we have $\phi([n])\subset\{r_1,r_2,\dots, r_{\alpha_{j+1}}, b_1, b_2, \dots, b_{\beta_{j+1}}\}$. This leads to the desired bound:

\[\frac{|V(H')\cap[n]|}{n}\leq 1-\frac{(1-o(1))j}{n}\leq 1-\frac{(1-o(1))j}{\alpha_{j+1}+\beta_{j+1}}\leq 1-\frac{1-o(1)}{\left(\frac{1-\gamma}{2}h(\gamma)+\epsilon+1\right)\frac{2\lambda}{1+\lambda}}\] which for $\epsilon$ small enough and $n$ large enough can take values arbitrarily close to $f(\lambda)$.

The case in which $H'$ is monochromatic blue is analogous. Indeed, besides the direction of the rounding, it is equivalent to taking the function $-g(x)$ instead of $g(x)$.
\end{proof}
\section{General lower bounds}\label{sec:genlbounds}

In this section we will prove three lower bounds. One is item \ref{compi} from Theorem \ref{compmt}, another is the lower bound of item \ref{compiii} in the same theorem, and the final one is the following, which will be used in the proof of Theorem \ref{mainthmbip} and Theorem \ref{factornb}:

\begin{theo}\label{mainlb} Let $H$ be a locally finite graph, $a,b,r,s$ be positive integers with $a>b$, and $\Psi:V(H)\rightarrow[a]$ be a proper coloring. Suppose that there exist infinitely many pairwise disjoint doubly independent sets $I_1, I_2, \dots$ in $H$, each contained in some component of $H$ and not concentrated in fewer than $b$ components, such that $|I_i|=r$, $|N(I_i)|\leq s$, and $\Psi(N(I_i))=a$. Then \[\rho(H)\geq\frac{b}{a-1}f\left(\frac sr\right).\]\end{theo}

\begin{proof}[Proof of Theorem \ref{compmt}\ref{compi}]Let $\chi:E(K_\N)\rightarrow \{R,B\}$ be an edge-coloring. Let $\mathcal{F}$ be an inclusion-maximal family of pairwise disjoint monochromatic infinite cliques in $\chi$. Then $\N\setminus V(\mathcal{F})$ is finite, because otherwise by Ramsey's theorem there would be an infinite monochromatic clique in $\chi$ restricted to $\N\setminus V(\mathcal{F})$, contradicting the maximality of $\mathcal{F}$. Let $\mathcal{F}_R$ and $\mathcal{F}_B$ be the families of red and blue cliques in $\mathcal{F}$. Since $\bar d(V(\mathcal{F}_R)\cup V(\mathcal{F}_B))=1$, we have $\max\{\bar d(V(\mathcal{F}_R)),\bar d(V(\mathcal{F}_B))\}\geq 1/2$. Wlog assume $\bar d(\mathcal{F}_R))\geq 1/2$. We can suppose that $\mathcal{F}_R$ contains infinitely many cliques, because otherwise we can take one clique $K\in \mathcal{F}_R$ and divide it into infinitely many infinite cliques. Let $K_1, K_2, \dots$, be the cliques in $\mathcal{F}_R$. We can partition the vertex set of $H$ into infinitely many parts $S_1, S_2, \dots$, each of which is made up of infinitely many components of $H$. Now take any $\Phi:V(H)\rightarrow V(K_\N)$ which is a bijection from each $S_i$ to each $K_i$. The image of $H$ is a monochromatic graph $H'$ and $\bar d(V(H'))=\bar d(\mathcal{F}_R)\geq 1/2$.\end{proof}

The proof of Theorem \ref{compmt}\ref{compiii} and Theorem \ref{mainlb} will both be (partially) algorithmic: given a coloring $\chi:E(K_\N)\rightarrow\{R,B\}$, we will define an algorithm that constructs a dense monochromatic copy of $H$. The algorithms will be similar, so we will first prove Theorem \ref{mainlb} and then explain how to adapt the proof to Theorem \ref{compmt}\ref{compiii}.


Let $H$ be as in Theorem \ref{mainlb}, and let $\chi:E(K_\N)\rightarrow \{R,B\}$. Our goal is to find a copy of $H$ in $K_\N$ with density at least $b/(a-1)f(s/r)$. In order to find such a copy of $H$, it will be helpful to also color the vertices of $K_\N$, in a way that encodes information about how the vertices are connected through red or blue edges. The following coloring is a variant of one due to Elekes et al. \cite{ESSS}. We denote by $N_C(v)$ the set of vertices connected to $v$ through an edge of color $C$.

\begin{defi}Let $\chi:E(K_\N)\rightarrow\{R,B\}$ be a coloring, and let $a$ be a positive integer. An $a$-good coloring of $V(K_\N)$ is a partition $\N=\cup_{i=1}^a(R_i\cup B_i)\cup X$ into $2a+1$ classes (some of which might be empty), where $X$ is finite, with the following properties:
\begin{itemize}
\item For every color $C\in \{R,B\}$, every $1\leq i\leq a-1$ and every nonempty finite subet $S\subseteq C_i$, the set $\left(\cap_{v\in S}N_C(v)\right)\cap C_i$ is infinite.
\item For every color $C\in\{R,B\}$, every $1\leq i\leq a-1$ and every nonempty finite subet $S\subseteq C_a\cup\left(\cup_{j=i+1}^{a-1}\bar C_j\right)$, the set $\left(\cap_{v\in S}N_C(v)\right)\cap C_i$ is infinite.
\end{itemize}\end{defi}

We call each class $R_i$ a shade of red and each class $B_i$ a shade of blue. $X$ can be seen as a residual set, which can be removed without affecting the density of the graph. The choice of $a$ is related to the chromatic number of the monochromatic subgraphs that we can find in this graph. Indeed, say that we want to find a red clique of size $a$ containing $v\in R_i$. If $i\leq a-1$, then we can set $v=v_1$, and then greedily select $v_2, v_3, \dots, v_a\in R_i$, each adjacent to the previous ones through a red edge. If $i=a$, we can set $v=v_a$, and then greedily select $v_{a-1}, v_{a-2}, \dots, v_1$, with $v_j\in B_j$, each adjacent to the previous ones through a red edge.

We denote by $K_{r,s}^C$ a complete bipartite graph in which all edges have color $C$, all vertices in the part of size $s$ have color $C$ and all vertices in the part of size $r$ have color $\bar C$. These subgraphs will be used to embed the sets $I_i\cup N(I_i)$ in our colored graph.

The proof of Theorem \ref{mainlb} will have three main steps, which are captured by these lemmas:

\begin{lemma}\label{essscol} Let $\chi:E(K_\N)\rightarrow\{R,B\}$ be a coloring, and let $a$ be a positive integer. There exists an $a$-good coloring in which  at least two of $(R_a\cup B_{a-1})$, $(B_a\cup R_{a-1})$ and $X$ are empty.
\end{lemma}

\begin{lemma}\label{finddis} Let $\chi:E(K_\N)\cup V(K_\N)$ be a coloring, and let $r,s$ be positive integers. There exists a color $C$ and a subgraph $W\subseteq K_\N$, with $\bar d(W)\geq f(s/r)$, in which every component is either an isolated vertex with color $C$, or a $K_{r,s}^C$. Furthermore, if $V(K_\N)$ is further subdivided into finitely many shades, then $W$ can be taken in a way that each $K_{r,s}^C$ only uses one shade of each color. \end{lemma}

\begin{lemma}\label{algoh} Let $\chi:E(K_\N)\rightarrow\{R,B\}$ be an edge-coloring, let $a\geq a'\geq b$ be positive integers. Let $\N\rightarrow \{R_1, \dots, R_a, B_1, \dots, B_a, X\}$ be an $a$-good coloring in which at most $a'$ shades of each color are non-empty. Let $W\subseteq K_\N$ be a subgraph in which every component is either an isolated vertex with color $C$, or a $K_{r,s}^C$ which uses only one shade of each color. Under the conditions of Theorem \ref{mainlb}, there exists a monochromatic $H'\subseteq K_\N$ of color $C$, $H'\simeq H$, with $\bar d(H')\geq b/a'\bar d(W)$.\end{lemma}

It is straightforward to combine these three lemmas to deduce Theorem \ref{mainlb}:

\begin{proof}[Proof of Theorem \ref{mainlb}] Let $\chi:E(K_\N)$ be given. Apply Lemma \ref{essscol} to this edge-coloring to obtain an $a$-good coloring with at most $a-1$ shades of each color are non-empty. Assign the color red to the vertices in $X$. Apply Lemma \ref{finddis} to obtain $C$ and $W$. Remove from $W$ every component which uses a vertex of $X$ (this does not affect $\bar d(W)$ because it only removes finitely many vertices). By Lemma \ref{algoh}, we can find a monochromatic $H'\subseteq K_\N$ with $\bar d(H')\geq b/(a-1)\bar d(W)\geq b/(a-1)f(s/r)$.\end{proof}

\begin{proof}[Proof of Lemma \ref{essscol}]
For each vertex $v$, we will denote by $c(v)$ and $s(v)$ the color that we assign to it, respectively. The color assigned to a vertex might change while the algorithm is running, but the shade of each vertex is final once assigned and it will match the color that the vertex has at that time. 

At some points, the shade assigning algorithm will call the basic coloring algorithm to color an infinite set $V=\{v_1, v_2, \dots\}$ of vertices. We will first describe this algorithm. First, the color $c(v_1)$ is assigned so that $N_{c(v_1)}(v_1)\cap V$ is infinite. Once the colors of $v_1, \dots, v_{i-1}$ have been assigned, assuming that $\left(\cap_{i=1}^{n-1}N_{c(v_i)}(v_i)\right)\cap V$ is infinite, the color $c(v_n)$ is chosen so that $\left(\cap_{i=1}^{n}N_{c(v_i)}(v_i)\right)\cap V$ is infinite.

The coloring produced satisfies that $\left(\cap_{i=1}^{n}N_{c(v_i)}(v_i)\right)\cap V$ is infinite for every $n$. We say that a color $C$ is dominant in this coloring if, for every $n$, $\left(\cap_{i=1}^{n}N_{c(v_i)}(v_i)\right)\cap V$ contains infinitely many vertices $v$ with $c(v)=C$. Observe that at least one of the colors is dominant.

Now we define the shade assigning algorithm:

\begin{enumerate}
\item For every $v\in\N$, start with $c(v)$ and $s(v)$ unassigned.
\item If finitely many vertices $v$ remain with $s(v)$ unassigned, assign $s(v)=X$, and END.
\item Let $V$ be the set of vertices without a shade. Color $V$ with the basic coloring algorithm. Choose a color $C$ that is dominant. Let $i$ be the minimum value such that $C_i$ is empty. For every $v\in V$ with $c(v)=C$, set $s(v)=C_i$.
\item If $i=a-1$, set $s(v)=\bar C_a$ for every $v\in V$ with $c(v)=\bar C$, and END. If $i\neq a-1$, return to Step 2.
\end{enumerate}

The algorithm runs the loop $2-4$ at most $2a-3$ times before ending. Whenever a set $C_i$ with $i\leq a-1$ is defined, the color $C$ is dominant in the corresponding coloring, meaning that in particular $\left(\cap_{v\in S}N_C(v)\right)\cap C_i$ is infinite for every finite non-empty $S\subseteq C_i$, as it is a superset of the color $C$ vertices of $\left(\cap_{i=1}^nN_{c(v_i)}(v_i)\right)\cap V$ for $n$ large enough. For the same reason, for any finite subset $S$ of vertices whose shade is not assigned when $C_i$ is defined, we have that $\left(\cap_{v\in S}N_{\bar C}(v)\right)\cap C_i$ is infinite. If $C_a$ is defined at some point in the algorithm (namely at the end), then $\bar C_1, \bar C_2, \dots, \bar C_{a-1}, C_a$ are defined in this order. This proves that the coloring that we obtained is $a$-good.

To conclude the proof of Lemma \ref{essscol}, simply observe that $X$ is nonempty only if the algorithm terminates at Step 2, the set $(R_a\cup B_{a-1})$ is nonempty only if the algorithm terminates at Step 4 with $C=B$ and $(B_a\cup R_{a-1})$ is nonempty only if the algorithm terminates at Step 4 with $C=R$.\end{proof}

The proof of Lemma \ref{finddis} divides $K_\N$ into infinitely many finite graphs, and then combines the regularity lemma and a max flow/min cut argument, to reduce the problem to an optimization problem equivalent to \eqref{optiprob}. We will now state the lemmas that we will need for this:

\begin{lemma}[Regularity Lemma \cite{KS96}]\label{regularity}
	For every $\epsilon>0$ and $m_0, \ell\geq 1$ there exists $M = M(\epsilon,m_0,\ell)$ such that the following holds. Let $G$ be a graph on $n  \geq M $ vertices whose edges are coloured in red and blue and let $d>0$.
	Let $\{W_i\}_{i \in [\ell]}$ be a partition of $V(G)$. Then there exists a partition $\{V_0, \dots, V_m\}$ of $V(G)$ and a subgraph $H$ of $G$ with vertex set $V(G) \setminus V_0$ such that the following holds:
	\begin{enumerate}
		\item $m_0 \leq m \leq M$;
		\item $\{V_i\}_{i \in [m]}$ refines $\{W_i\cap V(H)\}_{i \in [\ell]}$;
		\item $|V_0| \leq \epsilon n$ and $|V_1| = \dots = |V_m| \leq \lceil\epsilon n \rceil$;
		\item $\deg_{H}(v) \geq \deg_G(v)-(d+\epsilon)n$ for each $v \in V(G) \setminus V_0$;
		\item $H[V_i] $ has no edges for $i \in [m]$;
		\item all pairs $(V_i,V_j)$ are $\epsilon$-regular and with density either 0 or at least $d$ in each colour in $H$.
	\end{enumerate}
\end{lemma}

The max flow-min cut result that we will use can be seen as a weighted version of K\"onig's Theorem:

\begin{lemma}\label{mfmc} Let $G$ be a finite bipartite graph on $V=(X,Y)$, and let $r,s$ be positive integers. There exists a unique value of $D$ for which both of these exist:
\begin{itemize}
\item A function $h:E(G)\rightarrow \N\cup\{0\}$ such that $\sum_{e\ni v}h(e)\leq r$ if $v\in X$, $\sum_{e\ni v}h(e)\leq s$ if $v\in Y$ and $\sum_{e\in E(G)}h(e)=D$.
\item A vertex cover $Z$ of $G$ such that $r|Z\cap X|+s|Z\cap Y|=D$.
\end{itemize}\end{lemma}

\begin{proof}
Take an orientation of every edge in $G$ from $X$ to $Y$, and give it an infinite capacity. Connect every vertex in $X$ to a source $\sigma$ through an edge with capacity $r$, and every vertex in $y$ to a sink $\tau$ through an edge with capacity $s$. Let $D$ be the maximum flow in this network. $D$ is the maximum value for which a function $h$ as in the statement exists (by the integrality theorem, there exists a maximum flow in which the flow of every edge is an integer). $D$ is also the minimum value for which a cut $(C_1, C_2)$ with $\sigma\in C_1$ and $\tau\in C_2$ exists. Observe that $(C_1, C_2)$ is a cut with finite capacity iff $(C_2\cap X)\cup (C_1\cap Y)$ is a vertex cover of $G$, in which case the capacity of the cut is $r|C_2\cap X|+s|C_1\cap Y|$. Our lemma follows from the Ford-Fulkerson theorem.
\end{proof}

The next lemma that we will introduce requires the definition of two parameters, which up to a change of coordinates are equivalent to $\Gamma^+_\gamma$ and $\Gamma^-_\gamma$.
\begin{defi} Let $g:[0,+\infty)\rightarrow[0,+\infty)$ be a continuous, non-decreasing function. Let $\lambda, t$ be positive real numbers. We define the following two parameters: \[\ell_\lambda^+(g,t)=\min\left\{x:g(\lambda x)-x\geq t\right\}\hskip 1cm \ell_\lambda^-(g,t)=\min\left\{x:x-\frac{g(x)}{\lambda}\geq t\right\},\]where we take the minimum of the empty set to be $+\infty$.\end{defi}

\begin{lemma}\label{asymt} For $\lambda,\epsilon>0$ there exists $\gamma>0$ with the following property: for every non-decreasing continuous function $g:[0,+\infty)\rightarrow [0, +\infty)$ with $g(0)=0$ and every $m>0$ there exists $t\in[\gamma m, m]$ such that \[\frac{\ell_\lambda^+(g,t)+\ell_\lambda^-(g,t)}{t}\geq\frac{f(\lambda)}{1-f(\lambda)}-\epsilon.\] \end{lemma}

The proof of Lemma \ref{asymt} can be found in the Appendix. Combining Lemma \ref{mfmc} and Lemma \ref{asymt}, we can obtain the following:

\begin{lemma}\label{findflow} For every $\epsilon, r,s>0$ there exists $\gamma,\eta>0$ and $N$ for which the following hold: for every graph $G$ on $[n]$, with $n>N$ and $\delta(G)\geq (1-\eta)n$, and for every total coloring $\chi:V(G)\cup E(G)\rightarrow \{R,B\}$, there exists $t\in[\gamma n,n]$, a color $C$, and $h:E(G)\rightarrow \N\cup\{0\}$, such that the following hold: 
\begin{itemize}
\item For every edge $e=uv$, if $g(e)>0$ then $\chi(e)=C$ and $\chi(u)\neq \chi(v)$.
\item $\sum\limits_{e\ni v}h(e)\leq r$ for every $v$ with $\chi(v)=C$ and $\sum\limits_{e\ni v}h(e)\leq s$ for every $v$ with $\chi(v)=\bar C$.
\item $\frac{|C\cap [t]|}t+\frac{\sum_{v\in (\bar C\cap[t])}\sum_{e\ni v} h(e)}{st}\geq f(s/r)-\epsilon$
\end{itemize} \end{lemma}

\begin{proof}

Let $\lambda=s/r$. For every red vertex $v$, we define its blue degree $d_B(v)$ as the number of blue vertices $w$ such that $vw$ is blue. Let $v_1, v_2, \dots, v_{|R|}$ be the set of red vertices, sorted from smallest to largest blue degree, and let $d_i=d_B(v_i)$. Define additionally $d_0=0$ and $d_k=d_{|R|}$ for $k>|R|$. Let $g:[0,+\infty)\rightarrow[0,+\infty)$ be the function that satisfies $g(k)=d_k$ for every integer $k$ and which is linear between every pair of consecutive integers. 

From Lemma \ref{asymt}, we can find $\gamma_1$, depending only on $\lambda$ and $\epsilon$, such that there exists $\tau\in \left[\gamma_1\frac{1-f(\lambda)}{1+\frac{\lambda}{1-f(\lambda)}}n, \frac{1-f(\lambda)}{1+\frac{\lambda}{1-f(\lambda)}}n\right]$ for which $\frac{\ell^+_\lambda(g,\tau)+\ell^-_\lambda(g,\tau)}{\tau}\geq\frac{f(\lambda)}{1-f(\lambda)}-\epsilon$. Let $t=\left(\frac{1}{1-f(\lambda)}-\epsilon\right)\tau$. Then either $|R\cap[t]|< \ell^-_\lambda(g,\tau)$ or $|B\cap[t]|\leq \ell^+_\lambda(g,\tau)+\tau$. We consider both cases:

\textbf{Case 1:} $|R\cap [t]|<\ell^-(g,\tau)$. Let $R'=R\cap[t]$. Let $G'$ be the graph of blue edges in $G$ between $R'$ and $B$. Let $h$, $Z$ and $D$ be as in Lemma \ref{mfmc} applied to $G'$, with $X=B$ and $Y=R'$. Suppose that $D\leq s(|R'|-\tau)$. Every vertex $v\in R'\setminus Z$ must have all its blue neighbors in $B\cap Z$, and so $d_B(v)\leq |B\cap Z|$. Therefore 
\[d_{|R'|- |Z\cap R'|}\leq |Z\cap B|=\frac{D-s|Z\cap R'|}{r}\leq \frac sr(|R'|-|Z\cap R'|-\tau).\] 
Setting $x=|R'|-|Z\cap R'|$, this expression rearranges to $x-\frac{g(x)}\lambda\leq \tau$, so by definition of $\ell^-_\lambda$ this means that $x\geq \ell^-_\lambda(g,\tau)$. But this is a contradiction, because $x\leq|R'|<\ell^-_\lambda(g,\tau)$. This means that we have $D> s(|R'|-\tau)$, and
\[\frac{|B\cap [t]|}{t}+\frac{D}{st}\geq \frac{t-|R'|}{t}+\frac{s(|R'|-\tau)}{st}=1-\frac\tau t=1-\frac{1}{\frac{1}{1-f(\lambda)}-\epsilon}\geq f(\lambda)-\epsilon.\]

\textbf{Case 2:} $|B\cap [t]|\leq\ell^+(g, \tau)+\tau$. Let $B'=B\cap[t]$. Let $G'$ be the graph of red edges between $R$ and $B'$. Let $h$, $Z$ and $D$ be as in Lemma \ref{mfmc} applied to $G'$, with $X=R$ and $Y=B'$. Suppose that $D<s(|B'|-\tau-\eta n-\frac1\lambda)$. Every edge between $R\setminus Z$ and $B'\setminus Z$ is blue. Every vertex $v$ has at most $\eta n$ vertices to which it is not connected, and so $d_B(v)\geq |B'\setminus Z|-\eta n$ for all\footnote{What if $R\setminus Z= \emptyset$ (if this happens we cannot guarantee $d_{|R\cap Z|+1}\geq |B'\setminus Z|$)? Then $r|R|\leq D\leq s|B'|\leq st\leq s\frac{\tau}{1-f(\lambda)}\leq s\frac{1-f(\lambda)}{\lambda}n\Rightarrow |R|\leq(1-f(\lambda))n\Rightarrow |B|\geq f(\lambda)n$, and thus taking $t'=n$, $h=0$ and $C=B$ is enough for Lemma \ref{findflow}.} $v\in R\setminus Z$. 
\begin{align*}d_{|R\cap Z|+1}\geq& |B'|-|B'\cap Z|-\eta n\geq |B'|-\frac{D-r|R\cap Z|}{s}-\eta n\\ =&\frac{s|B'|-D}{s}+\frac1\lambda|R\cap Z|-\eta n\geq\tau+\eta n+\frac1\lambda+\frac1\lambda|R\cap Z|-\eta n\\ \geq&\tau+\frac1\lambda(|R\cap Z|+1).\end{align*} 
Setting $x=\frac1\lambda(|R\cap Z|+1)$, this expression rearranges to $g(\lambda x)-x\geq\tau$, so by definition of $\ell^+_\lambda$ this means that $x\geq\ell^+_\lambda(g,\tau)$. On the other hand, $x=\frac{|R\cap Z|+1}{\lambda}\leq\frac{D}{s}+\frac1\lambda<|B'|-\tau-\eta n-\frac1\lambda+\frac1\lambda<|B'|-\tau\leq \ell^+_\lambda(g,\tau)$, which is a contradiction. This means that we have $D\geq s(|B'|-\tau-\eta n-\frac1\lambda)$, and \[\frac{|R\cap [t]|}{t}+\frac{D}{st}\geq\frac{t-|B'|}{t}+\frac{s(|B'|-\tau-\eta n-\frac1\lambda)}{st}\geq 1-\frac\tau t-\frac\eta\gamma-\frac1{\lambda\gamma n}\geq f(\lambda)-\epsilon\] for $\gamma=\gamma_1\frac{1-f(\lambda)}{1+\frac{\lambda}{1-f(\lambda)}}$, $\eta<<\gamma\epsilon$ and $n>N>>\frac{1}{\gamma\lambda\epsilon}$.\end{proof}

To prove Lemma \ref{finddis}, we apply the regularity lemma to the graph and use Lemma \ref{findflow}. We also use the fact that, by the K\H{o}v\'ari-S\'os-Tur\'an theorem, every large enough dense bipartite graph contains a large complete bipartite subgraph:

\begin{proof}[Proof of Lemma \ref{finddis}] We first claim that, for every $\epsilon>0$, there exists $\gamma(\epsilon)>0$ and $N(\epsilon)$ such that, for every $n>N$, there exist $t\in [\gamma n,n]$, a color $C$ and a subgraph $\mathcal{F}\subseteq K_\N$ contained in $[n]$ in which every component is either an isolated vertex of color $C$ or a $K_{r,s}^C$ using only a shade of each color, with \[\frac{|V(\mathcal{F})\cap [t]|}{t}\geq f\left(\frac sr\right)-\epsilon.\]

Fix $\epsilon>0$. Let $G$ be the restriction of our coloring to $K_\N$. Let $a$ be the total number of shades (from both colors). Let $\rho,\delta>0$ be very small real numbers, whose value we will define later. Take a partition of $[n]$ into $\ell=a\lceil\rho^{-1}\rceil$ parts $\{Z_1, \dots, Z_\ell\}$, such that each $Z_i$ is contained in one shade, and $\max Z_i-\min Z_i<\rho n$. Applying Lemma \ref{regularity} to $G$ with $d=2\delta$, we find $M(\delta, \ell, \ell)$, a subgraph $H\subseteq G$ and a partition $[n]=\{V_0, V_1, \dots, V_m\}$, with $\ell\leq m\leq M$, as in the statement of Lemma \ref{regularity}, replacing $\epsilon$ with $\delta$.

We suppose that the labeling of the parts is such that $\min V_1<\min V_2<\dots<\min V_m$. We define an auxilliary graph $H'$ as follows: the vertex set is $[m]$. The color of every vertex $i$ is the same as the color of each of its vertices in $G$. Between any two vertices $ij$, we draw an edge if the bipartite graph $V_iV_j$ is nonempty in $H$, and we color it in the most dense color in $V_iV_j$.

Let $y=|V_1|=\dots=|V_m|$. Then $\frac{(1-\delta)n}{m}\leq y\leq \frac{n}{m}$. The minimum degree in $H'$ is at least $(1-\frac{4\delta}{1-\delta})m$. Indeed, given $i$ and $v\in V_i$, we have $d_{H'}(i)\geq\frac{d_H(v)-\delta n}{y}\geq \frac{d_{G}(v)-4\delta n}{y}\geq(1-\frac{4\delta}{1-\delta})m$.

Apply Lemma \ref{findflow} to $H'$, with parameters $\epsilon/2, r,s$ to obtain $\gamma, \eta>0$, $N'$, $\tau\in [\gamma m,m]$, a color $C$ and a function $h:E(H')\rightarrow \N$ as in the statement of Lemma \ref{findflow}, replacing $t$ with $\tau$. Our value of $\rho$ will be chosen so that $\ell>N'$, and $\delta$ will be chosen so that $\frac{4\delta}{1-\delta}<\eta$ (note that the values of $\eta$ and $\gamma$ depend only on $\epsilon, r$ and $s$). Subdivide each  $V_i$ with color $C$ into $r$ parts $V_{i,1},\dots, V_{i,r}$, each of size at least $\lfloor y/r\rfloor$, and each $V_i$ with color $\bar C$ into $s$ parts $V_{i,1},\dots, V_{i,s}$, each of size at least $\lfloor y/s\rfloor$. Construct a matching $\mathcal{M}$ of pairs $(V_{i,k}, V_{j,k'})$, where for any fixed values of $i$ and $j$, the number of pairs $(V_{i,k},V_{j,k'})$ in $\mathcal{M}$ is $h(ij)$.

Within each pair $(V_{i,k}, V_{j,k'})$, where $V_i$ has color $C$ and $V_j$ has color $\bar C$, find a maximum family $\mathcal{F}_{i,k,j,k'}$ of disjoint copies of $K_{r,s}^C$. If $N$ is large enough compared to $\delta^{-1}$ and $M$, and therefore $\delta y$ is large enough, then $\min\{|V_{i,k}\setminus V(\mathcal{F}_{i,k,j,k'})|,|V_{j,k'}\setminus V(\mathcal{F}_{i,k,j,k'})|\}<\delta y$. That is because otherwise the bipartite graph between $V_{i,k}\setminus V(\mathcal{F}_{i,k,j,k'})$ and $V_{j,k'}\setminus V(\mathcal{F}_{i,k,j,k'})$ would have density at least $
\delta$ in the edges of color $C$, and for $\delta y$ large enough this implies the existence of a copy of $K_{r,s}^C$, which would contradict the maximality of $\mathcal{F}_{i,k,j,k'}$.

Let $\mathcal{F}$ be the union of all families $\mathcal{F}_{i,k,j,k'}$. Let $t=\min V_\tau$. We will now bound $\frac{|(V(\mathcal{F})\cup C)\cap[t]|}{[t]}$. If $v\geq t+\rho n$, and $v\in V_i$ with $i\neq 0$, then $\min V_i>\max V_i-\rho n\geq v-\rho n\geq t=\min V_\tau$, and thus $i>\tau$. This means that $|(\cup_{i=1}^\tau V_i)\setminus[t]|\leq\rho n$, and $t\geq \tau y-\rho n\geq \frac{(1-\delta)\tau n}{m}-\rho n$. On the other hand, if $v\leq t$ then either $v\in V_0$ or $v\in V_i$ with $\min V_i\leq v\leq t=\min V_\tau$, and thus $i\leq \tau$. This implies that $t\leq \sum_{i=0}^\tau|V_i|\leq \delta n+\tau y\leq \delta n+\frac{\tau n}m$. 

Every $V_i$ with color $C$ and $i\in [\tau]$ will trivially be contained in $(V(\mathcal F)\cup C)\cap (\cup_{i=1}^\tau V_i)$. For any $V_i$ with color $\bar C$ and $i\in[\tau]$, there are $\sum_{e\ni i}h(e)$ parts $V_{i,k}$ which are paired up with a different part $V_{j,k'}$. We either have $|V_{i,k}\setminus V(\mathcal{F})|\leq\delta y$ or $|V_{j,k'}\setminus V(\mathcal{F})|\leq\delta y$. In the first case, $|V_{i,k}\cap V(\mathcal{F})|\geq \lfloor y/s\rfloor-\delta y\geq (1/s-1/y-\delta)y$. In the second case, $|V_{j,k'}\cap V(\mathcal{F})|\geq \lfloor y/r\rfloor-\delta y\geq(1/r-1/y-\delta)y$. But $\mathcal{F}$ is a family of copies of $K_{r,s}$, so $|V_{i,k}\cap V(\mathcal{F})|=\frac rs|V_{j,k'}\cap V(\mathcal{F})|\geq (1/s-\lambda^{-1}(1/y+\delta))y$. In either case we have $|V_{i,k}\cap V(\mathcal{F})|\geq (1-\sigma(1/y+\delta))y/s$, for $\sigma=\max\{r,s\}$.

Putting our bounds together:

\begin{align*}\frac{|(V(\mathcal{F})\cup C_{G})\cap[t]|}{t}\geq & \frac{|(V(\mathcal{F})\cup C_{G})\cap(\cup_{i=1}^\tau V_i)|-\rho n}{t}\\
 \geq&\frac{y|C_{H'}\cap [\tau]|}{t}+\left(1-\sigma\left(\frac1y+\delta\right)\right)\frac ys\frac{\sum_{v\in (\bar C_{H'}\cap[\tau])}\sum_{e\ni v}h(e)}{t}-\frac{\rho n}{t}\\
 \geq&\left(1-\sigma\left(\frac1y+\delta\right)\right)\frac{\tau y}{t}\left(\frac{|C_{H'}\cap[\tau]|}{\tau}+\frac{\sum_{v\in (\bar C_{H'}\cap[\tau])}\sum_{e\ni v}h(e)}{s\tau}\right)-\frac{\rho n}{t}\\
 \geq&\left(1-\sigma\left(\frac1y+\delta\right)\right)\frac{\tau y}{t}\left(f(\lambda)-\frac\epsilon 2\right)-\frac{\rho}{\frac{\tau(1-\delta)}{m}-\rho}\\
 \geq& \left(1-\sigma\left(\frac1y+\delta\right)\right)\frac{\tau y}{\delta n+\tau y}\left(f(\lambda)-\frac\epsilon 2\right)-\frac{\rho}{\gamma(1-\delta)-\rho}\\
 \geq& \left(1-\sigma\left(\frac{m}{(1-\delta)n}+\delta\right)\right)\frac{1}{1+\delta\frac{n}{my}\frac m\tau}\left(f(\lambda)-\frac\epsilon 2\right)-\frac{\rho}{\gamma(1-\delta)-\rho}\\
 \geq& \left(1-\sigma\left(\frac{M}{(1-\delta)N}+\delta\right)\right)\frac{1}{1+\frac{\delta}{(1-\delta)\gamma}}\left(f(\lambda)-\frac\epsilon 2\right)-\frac{\rho}{\gamma(1-\delta)-\rho}\\
 \geq& f(\lambda)-\epsilon
 \end{align*}

if $\rho,\delta<<\epsilon, r,s$ and $N>>M$. To conclude the proof of our initial claim, notice that $t\geq \left(\frac \tau m-\rho\right)n\geq(\gamma-\rho) n\geq \gamma'n$ for a constant $\gamma'>0$. 

We are now ready to construct $W$. Take a sequence $f(s/r)>\epsilon_1> \epsilon_2> \dots>0$ with $\epsilon_i\rightarrow 0$. Start by applying the claim with $\epsilon=\epsilon_1$ and $n_1=N(\epsilon)$ to obtain a subgraph $\mathcal{F}_1$ with color $C_1$ with density at least $f(s/r)-\epsilon_1$ in $[t_1]$. Now proceed by induction, and set $n_i=\max\{N(\epsilon_i/2), 2n_{i-1}(r+s)/(\epsilon_i\gamma(\epsilon_i/2))\}$. Applying the claim with $\epsilon=\epsilon_i/2$ we find a subgraph $\mathcal{F}_i'$ with color $C_i$ contained in $[n_i]$ and with density at least $f(s/r)-\epsilon_i/2$ in $[t_i]$, for some $t_i\in[\gamma(\epsilon)n,n]$. Remove from $\mathcal{F}_i'$ all components that intersect $[n_{i-1}]$ (this represents at most $n_{i-1}(r+s)$ vertices) to obtain $\mathcal{F}_i$. Then $\mathcal{F}_i$ is disjoint from all previous $\mathcal{F}_j$, and by the choice of $n_i$, it still has density at least $f(s_r)-\epsilon_i$ in $[t_i]$.

Select a color $C$ such that $C_i=C$ for infinitely many $i$. Let $W=\cup_{C_i=C}\mathcal{F}_i$. Then by construction $\bar d(W)\geq f(s/r)$, since the $t_i$ tend to infinity, and the components of $W$ are isolated vertices of color $C$ or $K_{r,s}^C$. This concludes the proof of Lemma \ref{finddis}.
\end{proof}

Finally, we prove Lemma \ref{algoh} by defining an algorithm that constructs a mo\-no\-chro\-matic $H'$. This algorithm uses enough components from $W$ (mapping to them either single vertices of $H$ or sets $I_i\cup N(I_i)$) to keep a fraction of its density, and takes advantage of the properties of the $a$-good coloring to map the remaining vertices of $H$.

\begin{proof}[Proof of Lemma \ref{algoh}]Without loss of generality, assume that $C$ is red, let $S_j$ denote the vertices in $W$ of shade $R_j$, plus the blue vertices contained in a copy of $K_{r,s}^R$ in $W$ in which the red side has shade $R_j$. Removing from $W$ the finite sets $S_j$ does not affect its density, so suppose that each $S_j$ is either empty or infinite. We will show that there exists a set $J$, of size $b$, such that $\bar d(\cup_{j\in J}S_j)\geq b/a'\bar d(W)$.

By definition of density, there exists a sequence $n_1<n_2<\dots$ of positive integers such that $|V(W)\cap [n_i]|/n_i\rightarrow \bar d(W)$. For each $i$ there exists a subset $J_i\subseteq[a]$ of $b$ indices such that \[\frac{|(\cup_{j\in J_i}S_j)\cap [n_i]|}{n_i}\geq \frac{b}{a'}\frac{|V(W)\cap [n_i]|}{n_i}.\] For infinitely many $i$, the set $J_i$ is the same, which we denote $J$. By taking an appropriate subsequence of $n_1, n_2, \dots$, we can suppose without loss of generality that $J_i=J$ for all $i$ and that $n_{i+1}/n_i\rightarrow\infty$. Let $\mathcal{F}_j$ is the union of components from $W$ contained in some $S_j$ with $j\in J$, which contain a vertex from $[n_i]$ but no vertex from $[n_{i-1}]$.
Let $\mathcal{I}=\{I_i, I_2, \dots\}$ be the family of doubly independent sets. We can suppose that the elements in $\mathcal{I}$ are such that the sets $I_i\cup N(I_i)$ are pairwise disjoint. Indeed, because $H$ is locally finite, each $I_i\cup N(I_i)$ intersects finitely many sets $I_j\cup N(I_j)$, so we can find an infinite subfamily $\mathcal{I}'$ by including into it only the sets $I_i$ such that $I_i\cup N(I_i)$ does not intersect a set $I_j\cup N(I_j)$ for some $j<i$ with $I_j\in\mathcal{I'}$. This does not change the components in which $\mathcal{I}$ is concentrated. 

Let $J'\subseteq J$ be the set of indices in $j$ for which $S_j$ is non-empty. We assign to each component $\mathcal C\subseteq H$ a number $\kappa(\mathcal{C})\in J'$, in such a way that for every $j\in J'$ there are infinitely many sets $I_i$ in components with $\kappa(\mathcal{C})=j$. Indeed, if finitely many components intersect $\mathcal{I}$, there are at least $b$ components that contain infinitely many elements of $\mathcal{I}$, so give different values of $\kappa(\mathcal{C})$ to $|J'|\leq b$ of them, whereas if there are infinitely many components that intersect $\mathcal{I}$, we can assign each value of $J'$ to infinitely many of them. The purpose of $\kappa(\mathcal{C})$ will be to identify the shade of red to be used in the vertices while embedding $\mathcal{C}$ in the red edges of $\chi$.

We will define an injective graph homomorphism $\Phi:H\rightarrow K_\N$ which maps edges to red edges, whose image contains $\mathcal{F}_i$ for infinitely many $i$. This is enough to prove Theorem \ref{mainlb}, because for infinitely many large enough values of $i$ we have \begin{align*}\frac{|\Phi(V(H))\cap [n_i]|}{n_i}\geq& \frac{|V(\mathcal{F}_i)\cap [n_i]|}{n_i}\geq \frac{b}{a'}\frac{|V(W)\cap[n_i]|}{n_i}-\frac{(r+s)n_{i-1}}{n_i}\\ \geq&\frac{b}{a'}\bar d(W)-o(1).
\end{align*}

We will define $\Phi$ in steps. On every step, we will define the image of finitely many vertices of $H$. After every step, the following conditions must hold. Let $u,v$ be two adjacent vertices in some component $\mathcal{C}$ of $H$, such that $\Phi(v)$ is defined and $\Phi(u)$ is not. Then:

\begin{itemize}\item If $\kappa(\mathcal{C})\neq a$, then $\Phi(v)\in R_{\kappa(\mathcal{C})}$.
\item If $\kappa(\mathcal{C})=a$ and $\Psi(v)=a$, then $\Phi(v)\in R_a$.
\item If $\kappa(\mathcal{C})=a$ and $\Psi(v)\neq a$, then $\Phi(v)\in B_{\Psi(v)}$ and $\Psi(u)<\Psi(v)$.
\end{itemize}

The algorithm will consist of two operations that alternate: defining the image of a vertex $v\in V(H)$ and adding some $\mathcal{F}_i$ to the image of $\Phi$. If we identify $V(H)$ with $\N$, and always apply the first operation to the least vertex $v$ with undefined $\Phi(v)$, at the end of the algorithm $\Phi(v)$ will be defined for every vertex in $V(H)$.

\textbf{Define the image of a vertex $v\in V(H)$:} Suppose first that $v\in\mathcal{C}$ and $\kappa(\mathcal{C})=k\neq a$. Let $w_1, \dots, w_q$ be the neighbors of $v$ which have $\Phi(w_i)$ defined. By our invariant, the vertices $\Phi(w_i)$ all have shade $R_k$, and therefore there are infinitely many vertices $x$ in shade $R_k$ which are connected to every $\Phi(w_i)$ through a red edge. Select one such $x$ which is not yet in the image of $\Phi$, and set $\Phi(v)=x$.

Now suppose that $\kappa(\mathcal{C})=a$. Let $\Psi(v)=k$. For every edge $uw$ of $H$, define an orientation $\overrightarrow{uw}$ such that $\Psi(u)<\Psi(w)$. Let $T$ be the set of vertices that can be reached from $v$ in this orientation. Because $T$ is connected, does not contain a path of length greater than $a$, and the degree of every vertex is finite, by K\"onig's lemma $T$ is finite. Also, $T$ does not have an oriented cycle. Observe that, by our invariant, if $\overrightarrow{uw}$ is an edge and $\Phi(u)$ is defined, then $\Phi(w)$ is defined.

Now define $\Phi(w)$ for every $w\in T$ for which the image is still undefined, in decreasing order of $\Psi(w)$. If $\Psi(w)=a$, choose an arbitrary vertex $x\in R_a$ which is not yet the image of any vertex and set $\Phi(w)=x$. If $\Psi(w)=k<a$, then for every $w'\in N^+(w)$ the image $\Phi(w')$ is defined and in $B_{k+1}\cup\dots\cup B_{a-1}\cup R_a$. By the properties of $a$-good colorings, there are infinitely many vertices\footnote{If $N^+(w)$ is empty, how do we know that $B_k$ is infinite? Because $\kappa(\mathcal{C})=a$, we know that $R_a$ is not empty. By the properties of $a$-good colorings, every vertex $y\in R_a$ has infinitely many red neighbors in $B_k$, and in particular $B_k$ is infinite.} $x\in B_k$ which are connected to every $\Phi(w')$ through a red edge. Choose one which is not yet in the image of $\Phi$, and set $\Phi(w)=x$.

\textbf{Add some set $\mathcal{F}_i$ to the image:} Select some $\mathcal{F}_i$ which is so far disjoint with the image of $\Phi$. For each $K_{r,s}^R$ component $Z\subseteq \mathcal{F}_i\cap S_j$, choose a doubly independent set $I\subseteq V(H)$ in a component $\mathcal{C}$ with $\kappa(\mathcal{C})=j$, such that no vertex from $I\cup N(I)$ has a defined image. If $V(Z)=X\cup Y$ with $|X|=r$ blue and $|Y|=s$ red, then set $\Phi$ to be bijective from $I$ to $X$, and injective from $N(I)$ to $Y$. The vertices $v$ of $\mathcal{F}_i\cap S_j$ that remain outside of the image at this point all have shade $R_{j}$. For each of them, choose a vertex $w$ with $\Psi(w)=a$ in a component $\mathcal{C}$ with $\kappa(\mathcal{C})=j$ (there are infinitely many of these vertices), whose image is not yet defined, and set $\Phi(w)=v$. After doing this for every vertex in $\mathcal{F}_{i}\cap S_j$ for every $j\in J'$, the set $\mathcal{F}_i$ is contained in the image.

After both steps are applied alternatingly infinitely many times, the image of $\Phi$ is a monochromatic red graph $H'\subseteq K_\N$ which contains infinitely many sets $\mathcal{F}_i$, and therefore $\bar d(H')\geq b/a'\bar d(W)$.\end{proof}

To prove the upper bound of Theorem \ref{compmt}\ref{compiii}, we just need the following variant of Lemma \ref{algoh}. The proof is then analogous to the proof of Theorem \ref{mainlb}, except we bypass completely the use of Lemma \ref{finddis}, and we only need that in the $a$-good coloring we have $\max\{\bar d(R), \bar d(B)\}\geq 1/2$.

\begin{lemma}Let $\chi:E(K_\N)\rightarrow\{R,B\}$ be an edge-coloring, let $a\geq a'\geq b$ be positive integers. Let $\N\rightarrow \{R_1, \dots, R_a, B_1, \dots, B_a, X\}$ be an $a$-good coloring in which at most $a'$ shades of each color are non-empty. Let $C\in\{R,B\}$. Let $H$ be a graph with chromatic number $a$ and at least $b$ infinite components. Then there exists a monochromatic $H'\subseteq K_\N$ of color $C$, $H'\simeq H$, with $\bar d(H')\geq b/a'\bar d(C)$.\end{lemma}

\begin{proof} Let $\Psi:V(H)\rightarrow[a]$ be a proper coloring, in which in every component of $H$ the most common color is $a$. Without loss of generality suppose that $C$ is red. As in the proof of Lemma \ref{algoh}, there exists $J'$ with $|J'|\leq b$, such that $\bar d(\cup_{j\in J'}R_j)\geq b/a'$, and all $R_j$ with $J\in J'$ are infinite. Let $\mathcal{F}=\cup_{j\in J'}R_j$. Define a function $\kappa$ from the components of $H$ to $J'$ for which the pre-image of every value contains infinitely many vertices. The algorithm now alternates between \textbf{define the image of a vertex $v\in V(H)$}, as above, and \textbf{add a vertex of $\mathcal{F}$ to the image}. At the end of the procedure, we obtain a red $H'\subseteq K_\N$ isomorphic to $H$ which contains $\mathcal{F}$, and thus has density at least $\bar d(\mathcal{F})\geq b/a'\bar d(R)$.

\textbf{Add a vertex of $\mathcal{F}$ to the image:} Let $v\in\mathcal{F}$ be a vertex in $R_j$. Choose a vertex $w$ in a component $\mathcal{C}$ with $\kappa(\mathcal{C})=j$, such that no vertex in $w\cup N(w)$ has a defined image and with $\Psi(w)=a$, and set $\Phi(w)=v$.\end{proof}

\section{Bounds for particular families of graphs}\label{sec:examples}

The goal of this section is to prove the remaining results from Section \ref{sec:intro}. We will start stating and proving the more general result for bipartite graphs, which implies Theorem \ref{forest}.

\begin{theo}\label{mainthmbip} Let $H$ be a locally finite bipartite graph, and let $\lambda=\liminf_{n\rightarrow\infty}\frac{\mu(H,n)}{n}$. Suppose that for every $\lambda'>\lambda$ there exist infinitely many pairwise disjoint independent sets $I_1, I_2, \dots$, all of the same size, with $\frac{|N(I_i)|}{|I_i|}\leq \lambda'$. Then $\rho(H)=f(\lambda)$.  \end{theo}

\begin{proof}The upper bound follows from Theorem \ref{mainub}. We will show that, for every $\epsilon>0$, we have $\rho(H)\geq f(\lambda)-\epsilon$. Our goal is to show that $H$ satisfies the condition of Theorem \ref{mainlb} for $a=2$, $b=1$, a certain coloring $\Psi$ and some doubly independent sets $I'_i$. Let $\Psi:V(H)\rightarrow\{1,2\}$ be a proper coloring. Choose $\lambda'>\lambda$ such that $f(\lambda')>f(\lambda)-\epsilon$ (it exists by continuity of $f$). There exist infinitely many pairwise disjoint independent sets $I_i$, all with the same size, such that $\frac{|N(I_i)|}{|I_i|}\leq \lambda'$ (by the condition from the statement). Partition each set $I_i$ into non-empty sets $I_{i,1}, \dots, I_{i,k_i}$, where each vertex $v$ is classified according to its color by $\Psi$ and the component it belongs to. If two vertices $v,w$ have a common neighbor, then they are in the same component and $\Psi(u)=\Psi(v)$. For this reason, $|N(I_i)|=\sum_{j=1}^{k_i}|N(I_{i,j})|$. There exists some $\tau_i$ such that \[\frac{|N(I_{i,\tau_i})|}{|I_{i,\tau_i}|}\leq\frac{\sum_{j=1}^{k_i}|N(I_{i,j})|}{\sum_{j=1}^{k_i}|I_{i,j}|}=\frac{|N(I_i)|}{|I_i|}\leq\lambda'\] Set $I'_i=I_{i,\tau_i}$. Set $r_i=|I'_i|$ and $s_i=|N(I'_i)|$. There is a pair $(r,s)$ satisfying $(r,s)=(r_i,s_i)$ for infinitely many values of $i$. Considering only the values of $i$ for which this equality holds, we have our set of independent sets. Note that, because $H$ is bipartite, $N(I'_i)$ is monochromatic and thus independent, meaning that $I'_i$ is doubly independent. If $\Psi(I'_i)=2$ does not hold for infinitely many $i$, replace $\Psi$ with $\bar\Psi=3-\Psi$. We can now apply Theorem \ref{mainlb} to obtain $\rho(H)\geq f(s/r)\geq f(\lambda')$. \end{proof}

To deduce Theorem \ref{forest} from Theorem \ref{mainthmbip}, we need to show that, in both graphs with infinite orbits and forests, the condition in the statement of Theorem \ref{mainthmbip} holds.

\begin{proof}[Proof of Theorem \ref{forest}] Let $\lambda=\liminf\limits_{n\rightarrow\infty}\frac{\mu(H,n)}{n}$. Fix $\lambda'>\lambda$. We will show that, in both cases, there exist infinitely many pairwise disjoint independent sets $I_1, I_2, \dots$, all with the same size, with $\frac{|N(I_i)|}{|I_i|}\leq\lambda'$. 

\textbf{For graphs with infinite orbits:} Choose $n$ such that $\frac{\mu(H,n)}{n}<\lambda'$. Let $I$ be an independent set of size $n$ with $|N(I)|=\mu(H,n)$. We will show that there are infinitely many automorphisms $\sigma_i\in {\rm Aut}(H)$ such that the sets $\sigma_i(I)$ are pairwise disjoint. Then we can take $I_i=\sigma_i(I)$ to conlude the proof. We proceed by induction on $n$. For $n=1$, if $I=\{v\}$, this is equivalent to the orbit of $v$ being infinite.

Suppose that the result is true for $n-1$. Suppose that we have already found $\sigma_1, \sigma_2, \dots, \sigma_k$ such that the sets $\sigma_i(I)$ are pairwise disjoint. Let $X=\cup_{i=1}^k\sigma_i(I)$. We will construct $\sigma_{k+1}\in {\rm Aut}(H)$ such that $\sigma_{k+1}(I)$ is disjoint from $X$. Choose $v\in I$. By the induction hypothesis, there is an infinite family $\{\tau_i\}_{i=1}^\infty\subseteq {\rm Aut}(H)$ such that the sets $\tau_i(I-v)$ are pairwise disjoint. If $\tau_i(v)\not\in X$ for some $i$, then we can take $\sigma_{k+1}=\tau_i$, and we are done. Therefore, assume that $\tau_i(v)\in X$ for every $i$. By pigeonhole principle, there exists $w$ such that $\tau_i(v)=w$ for infinitely many $i$. Choose $\phi\in {\rm Aut(H)}$ such that $\phi(w)\not\in X$ (it exists because the orbit of $w$ is infinite). The set $\phi^{-1}(X)$ intersects finitely many sets $\tau_i(I-v)$, therefore there exists some $i$ with $\tau_i(I-v)$ disjoint from $\phi^{-1}(X)$ and $\tau_i(v)=w$. Putting this together, $\phi(\tau_i(I))$ is disjoint from $X$, as we wanted.

\textbf{For forests:} The following lemma will be used to find independent sets of bounded size with bounded expansion within larger independent sets:

\begin{lemma}\label{treecut} For every $\lambda'>\lambda$ there exists $M=M(\lambda,\lambda')$ such that, for every independent set $I$ in a forest with $|N(I)|\leq\lambda|I|$, there exists $I'\subseteq I$ with $|N(I')|\leq \lambda'|I'|$ and $|I'|\leq M$.\end{lemma}

Knowing this lemma, choose $\lambda''<\lambda'''\in(\lambda,\lambda')$, and set $M=M(\lambda''',\lambda')$. Suppose that we have already constructed pairwise disjoint independent sets $I_1, I_2, \dots, I_k$ with $|I_i|\leq M$ and $|N(I_i)|\leq \lambda'|I_i|$. We will find a new set $I_{k+1}$, disjoint from the others. Let $S=\cup_{i=1}^kI_i$. There exists $n$ large enough such that $\frac{n}{n-|S|}\leq \frac{\lambda'''}{\lambda''}$. By definition of $\liminf$ and $\mu(H,n)$, there exists an independent set $I$ with $|I|\geq n$ and $|N(I)|\leq \lambda''|I|$. Then \[|N(I\setminus S)|\leq |N(I)|\leq \lambda''|I|\leq \lambda''(|I\setminus S|+|S|)\leq \lambda'''|I\setminus S|\]

By our claim, there exists $I_{i+1}\subseteq I\setminus S$ such that $|I_{k+1}|\leq M$ and $|N(I_{k+1})|\leq \lambda' |I_{k+1}|$. Once we have constructed an infinite family of independent sets $I_1, I_2, \dots$, simply take a pair $(r,s)$ which is equal to $(|I_i|, |N(I_i)|)$ for infinitely many $i$ (which is possible because this pair can only take finitely many values), and we are done.\end{proof}

\begin{proof}[Proof of Lemma \ref{treecut}]
Let $\delta=\delta(\lambda,\lambda')>0$ be small enough, which we will fix later. Let $F$ be the forest with vertex set $I\cup N(I)$ and containing only the edges between $I$ and $N(I)$ in our original graph. It is enough to prove our result in $F$. Denote $J=N(I)$. For every component of $F$ take a vertex of $I$ as the root.

There exists a set $S\subseteq V(F)$ with $|S|\leq\delta|V(F)|$, satisfying that every component of $F\setminus S$ has size at most $\delta^{-1}$. Indeed, start with $S=\emptyset$ and consider the set $U$ of vertices whose component in $F\setminus S$ contains at least $\delta^{-1}$ vertices. The rooted forest structure in $F$ induces a rooted forest structure in $F\setminus S$. Let $U'$ be the set of vertices in $V\setminus F$ which have at least $\delta^{-1}-1$ descendants. If $U\neq\emptyset$ then $U'\neq\emptyset$, because the root of the largest component will be in $U'$. Select a minimal vertex $v$ in $U'$, and add it to $S$. This removes $v$ and all its descendants from $U$, and thus reduces the size of $U$ by at least $\delta^{-1}$. After at most $\delta|V(F)|$ steps, $U$ will be empty.

Let $X$ be the union of $S$ and the parents of the vertices of $S\cap J$. This set has $|X|\leq 2|S|\leq 2\delta|V(F)|$, and every component of $F\setminus X$ is adjacent to at most one vertex in $X\cap J$, in which case it is the parent of the root. As a consequence, every component of $F\setminus (X\cap I)$ contains at most one vertex from $X\cap J$.

Let $\mathcal{C}=\{C_1, \dots, C_k\}$ be the components of $F\setminus (X\cap I)$. Then \[\frac{\sum_{j=1}^k|C_i\cap J|}{\sum_{j=1}^k|C_i\cap I|}= \frac{|J|}{|I|-|X\cap I|}\leq \frac{|N(I)|}{|I|-2\delta(|I|+|N(I)|)}\leq \frac{\lambda}{1-2\delta(1+\lambda)}=:\lambda''.\]

There exists some component $C_i$ such that $|C_i\cap J|\leq\lambda''|C_i\cap I|$. If $C_i\cap I$ has size not greater than $M:=2\delta^{-1}$, then set $I'=C_i\cap I$ and we are done, because $N(I')\subseteq C_i\cap J$. Otherwise $C_i$ has size greater than $2\delta^{-1}$, hence it must contain a (unique) vertex $v\in X\cap J$. Let $C'_1, C'_2, \dots, C'_r$ be the components obtained from $C_i$ by removing $v$, labeled in decreasing order of $|C'_j\cap J|/|C'_j\cap I|$. Consider the minimum integer $t$ such that $\sum_{j=1}^t|C'_j\cap I|\geq \delta^{-1}$. Because every component in $F\setminus X$ has size at most $\delta^{-1}$, we have $\sum_{j=1}^t|C'_j\cap I|\leq \sum_{j=1}^{t-1}|C'_j\cap I|+\delta^{-1}\leq2\delta^{-1}=M$. Set $I'=\cup_{j=1}^t(C'_i\cap I)$. Then \[\frac{|N(I')|}{|I'|}=\frac{1+\sum_{j=1}^t|C'_j\cap J|}{\sum_{j=1}^t|C'_j\cap I|}\leq  \delta+\frac{\sum_{j=1}^r|C'_j\cap J|}{\sum_{j=1}^r|C'_j\cap I|}\leq \delta+\lambda''.\]
This proves Lemma \ref{treecut}, for $\delta>0$ small enough such that $\delta+\lambda''<\lambda'$.\end{proof}

Next we will prove Corollaries \ref{kary} to \ref{bipfac} as direct applications of Theorem \ref{forest}:

\begin{proof}[Proof of Corollary \ref{kary}] We will show that $\mu(T_k,n)=kn$. For every independent set $I$ of size $n$, the set of children of the vertices of $I$ has size $kn$ and is contained in $N(I)$, thus $|N(I)|\geq kn$. Equality can hold, for example for $I=\{v_1, \dots, v_n\}$ where $v_1$ is the root of $T_k$ and $v_{i+1}$ is a grandchild of $v_i$. We therefore have $\mu(T_k,n)=kn$. Since $T_k$ is a forest, Theorem \ref{forest} applies and $\rho(T_k)=f(k)$.\end{proof}

\begin{proof}[Proof of Corollary \ref{grid}] Let $I$ be an independent set. The set $I+(0,1)$ is contained in $N(I)$, so $|N(I)|\geq|I|$ and $\mu(\mathrm{Grid}_d,n)\geq n$ for all $n$. On the other hand, let $I_k$ be the set of vertices in $[2k]^d$ with odd sum of coordinates. $I_k$ is an independent set of size $(2k)^d/2$, and $I\cup N(I)$ is contained in $[2k+2]^d$. Since $I$ and $N(I)$ are disjoint, \[\frac{|N(I)|}{|I|}=\frac{|I\cup N(I)|}{|I|}-1\leq \frac{(2k+2)^d}{(2k)^d/2}-1,\] which tends to 1 as $k\rightarrow\infty$. We have $\liminf\limits_{n\rightarrow\infty}\frac{\mu(\mathrm{Grid_d},n)}{n}=1$. The graph $\mathrm{Grid}_d$ is vertex-transitive, so by Theorem \ref{forest} we have $\rho(\mathrm{Grid}_d)=f(1)$.\end{proof}


Next we will deduce Theorem \ref{factornb} from Theorem \ref{mainlb}:

\begin{proof}[Proof of Theorem \ref{factornb}] Let $a=|V(F)|$, and let $b=a-1$. Let $\Psi:V(F)\rightarrow[a]$ be a coloring that assigns the value $a$ to every vertex in $N(I)$, and where the remaining vertices in $F$ all get different values in $[a-1]$. Because $I$ is doubly independent, this is a proper coloring. $\Psi$ extends to a coloring of $\omega\cdot F$, by coloring all copies of $F$ equally.

Let $I_1, I_2, \dots$ be the sets $I$ of all copies of $F$. Each $I_i$ is contained in a component of $F$, $\Psi(N(I_i))=a$ and the family of sets $I_i$ is not concentrated in fewer than $b$ components. Thus, by Theorem \ref{mainlb}, setting $r=|I|$ and $s=|N(I)|$, we have $\rho(\omega\cdot F)\geq f\left(\frac{|N(I)|}{|I|}\right)$. \end{proof}

Finally, we will prove Theorem \ref{beslb} using a result of Burr, Erd\H{o}s and Spencer~\cite{BurErdSpe} for the Ramsey number of $n\cdot F$:

\begin{theo}\label{buersp} Let $F_1, F_2$ be two finite graphs without isolated vertices. The two-color Ramsey number $R(n\cdot F_1, n\cdot F_2)$ satisfies \[R(n\cdot F_1, n\cdot F_2)=(|V(F_1)|+|V(F_2)|-\min\{\alpha(F_1),\alpha(F_2)\})n+O(1),\] where $\alpha(G)$ is the size of the largest independent set in $G$. In particular, $R(n\cdot F,n\cdot F)=(2|V(F)|-\alpha(F))n+O(1)$.\end{theo}

\begin{proof}[Proof of Theorem \ref{beslb}] Let $\chi: E(K_\N)\rightarrow \{R.B\}$ be a coloring. Let $n_1, n_2, \dots$ be an increasing sequence of positive integers with $n_{i+1}/n_i\rightarrow\infty$. Let $k_i$ be the maximum value such that $R(k_i\cdot F, k_i\cdot F)\leq n_{i+1}-n_i$. By Theorem \ref{buersp}, we have 
\[k_i=\left(\frac{1}{2|V(F)|-\alpha(F)}+o(1)\right)(n_{i+1}-n_i)=\left(\frac{1}{2|V(F)|-\alpha(F)}+o(1)\right)n_{i+1}.\]
There exist a family $\mathcal{F}_i$ of $k_i$ monochromatic disjoint copies of $F$ with vertices in $[n_i+1, n_{n+1}]$, all with the same color $C_i$. Choose a color $C$ which is equal to $C_i$ for infinitely many $i$. Then $H'=\cup_{C_i=C}\mathcal{F}_i$ is a copy of $\omega\cdot F$ with \[\limsup_{n\rightarrow\infty}\frac{|V(H)\cap [n]|}{n}\geq \limsup_{i:C_i=C}\frac{k_i|V(F)|}{n_{i+1}}=\frac{|V(F)|}{2|V(F)|-\alpha(F)}.\qedhere\]\end{proof}

\section{Open problems and questions}\label{sec:openprob}

Out of all the graphs $H$ for which the value of $\rho(H)$ is not known, the graph $\omega\cdot K_3$ presents an interesting challenge. A priori, the problem does not seem all that different from that of bounding $R(n\cdot K_3, n\cdot K_3)$, using the same finite-to-infinite trick that we used for Theorem \ref{beslb}. However, the coloring of the complete graph that achieves the finite bound does not extend naturally to a coloring of $K_\N$ without a dense $\omega\cdot K_3$. Additionally, it is hard to find triangles using the techniques of Theorem \ref{mainlb}, which look for structure in bipartite graphs.

\begin{prob} Improve either bound in $3/5\leq \rho(\omega\cdot K_3)\leq(21+\sqrt{12})/33\approx 0.74133$.\end{prob}

As noted in the introduction, $f$ depends on the solution of a certain optimization problem of Lipschitz functions. It would be helpful to remove such dependency, and obtain a closed formula for $f$ (perhaps the upper bound of \eqref{func}). In particular, if the upper bound of \eqref{func} is tight then we can observe that the behavior of $f$ changes at the value 3, which corresponds to $\limsup \mu(H,n)/H=3$ (the independent sets have similar expansion ratio as in the infinite ternary graph). The cause of this is that the optimal coloring of $K_\N$ changes.

\begin{prob} Find a closed formula for $f(x)$. In particular, prove or disprove that it matches the upper bound from \eqref{func}.\end{prob}

\bibliography{bibliog}
\bibliographystyle{abbrv}

\appendix
\section{Proof of Lemma \ref{asymt}}\label{annexlem}

We will show that the parameters from Lemma \ref{asymt} are essentially the same as $\Gamma^+$ and $\Gamma^-$, except the axes are rotated by 45 degrees. Thus we get that Lemma \ref{asymt} is equivalent to the following lemma, which is just a compacity argument away from completing the proof:

\begin{lemma}\label{rotat}
Let $\lambda\in(-1,1)$ and $\epsilon>0$. There exists $\gamma>0$ with the following property. For every 1-Lipschitz function $g:[0,+\infty)\rightarrow\mathbb{R}$ with $g(0)=0$, and every $m>0$, there exists $t\in[\gamma m,m]$ such that \[\Gamma^+_\lambda(g,t)+\Gamma^-_\lambda(g,t)\geq(h(\lambda)-\epsilon)t=\frac{2}{1-\lambda^2}\left(\frac{f\left(\frac{1+\lambda}{1-\lambda}\right)}{1-f\left(\frac{1+\lambda}{1-\lambda}\right)}-\lambda\right)t-\epsilon t.\]
\end{lemma}

\begin{proof}[Proof of Lemma \ref{asymt}]

Define the function $z:[0,+\infty)\rightarrow \mathbb{R}$ as follows: for every $x$, let $z(x)=g(y)-y$, where $y$ is the unique value such that $x=g(y)+y$. This function is 1-Lipschitz: if $x_1, x_2$ are non-negative, and $y_1,y_2$ are the corresponding values of $y$, then \begin{align*}|z(x_1)-z(x_2)|&=\left|\left(g(y_1)-g(y_2)\right)-\left(y_1-y_2\right)\right|\\ &\leq \left|\left(g(y_1)-g(y_2)\right)+\left(y_1-y_2\right)\right|\\ &=|x_1-x_2|,\end{align*} since $y_1-y_2$ and $g(y_1)-g(y_2)$ do not have opposite signs.

For every $t>0$, let $y_t=\ell^+_\lambda(g,t)$. By continuity of $g$, we have $ g(\lambda y_t)-y_t=t$. Let $x_t=g(\lambda y_t)+\lambda y_t$. Then \begin{align*}
z(x_t)=&g(\lambda y_t)-\lambda y_t\\
=&g(\lambda y_t)-\lambda y_t-\frac{2\lambda}{1+\lambda}\left(g(\lambda y_t)-y_t-t\right)\\
=&\frac{1-\lambda}{1+\lambda}\left(g(\lambda y_t)+\lambda y_t\right)+\frac{2\lambda}{1+\lambda}t\\
=&\frac{1-\lambda}{1+\lambda}x_t+\frac{2\lambda}{1+\lambda}t,
\end{align*}
which rearranges to $\frac{\lambda-1}{\lambda+1}x_t+z(x_t)=\frac{2\lambda}{1+\lambda}t$. Thus, $x_t\geq\Gamma^+_{\frac{\lambda-1}{\lambda+1}}(z,\frac{2\lambda}{1+\lambda}t)$. On the other hand, let $y'_t=\ell^-_\lambda(g,t)$. We have $y'_t-\frac{g(y'_t)}{\lambda}=t$. Let $x'_t=g(y'_t)+y'_t$. Then \begin{align*}
z(x'_t)=&g(y'_t)-y'_t\\
=&g(y'_t)-y'_t+\frac{2\lambda}{1+\lambda}\left(y'_t-\frac{g(y'_t)}{\lambda}-t\right)\\
=&\frac{\lambda-1}{\lambda+1}\left(g(y'_t)+y'_t\right)-\frac{2\lambda}{1+\lambda}t\\
=&\frac{\lambda-1}{\lambda+1}x'_t-\frac{2\lambda}{1+\lambda}t,
\end{align*}

which rearranges to $\frac{\lambda-1}{\lambda+1}x'_t-z(x'_t)=\frac{2\lambda}{1+\lambda}t$. Thus, $x'_t\geq \Gamma^-_{\frac{\lambda-1}{\lambda+1}}\left(z,\frac{2\lambda}{1+\lambda}t\right)$. By Lemma \ref{rotat}, for every $m$ there exists a value $t\in[\gamma m, m]$, where $\gamma$ depends only on $\lambda$ and $\delta$, such that

\begin{align*}\left(h\left(\frac{\lambda-1}{\lambda+1}\right)-\delta\right)\frac{2\lambda}{1+\lambda}t\leq&x_t+x'_t\\
=&g(\lambda y_t)+\lambda y_t+g(y'_t)+y'_t\\
=&(y_t+t)+\lambda y_t+\lambda(y'_t-t)+y'_t\\
=&(\lambda+1)(y_t+y'_t)+(1-\lambda)t.\end{align*}

Rearranging this inequality, substituting $h$ and taking $\delta$ small enough we obtain $y_t+y'_t\geq (f(\lambda)-\epsilon)t$.\end{proof}

Finally, Lemma \ref{rotat} can easily be proved by showing that the restriction $t\in[\gamma m,m]$ is unnecessary:

\begin{proof}[Proof of Lemma \ref{rotat}]
Suppose that Lemma \ref{rotat} is false. For some $\lambda\in(-1,1)$ there are 1-Lipschitz functions $g_1, g_2, \dots$ with $g_i(0)=0$ such that $\frac{\Gamma^+_\lambda(g_i,t)+\Gamma^-(g_i,t)}{t}<h(\lambda)-\epsilon$ for every $t\in [\frac{m_i}{i},m_i]$. By scaling each fraction as $\bar g_i(x)=\frac{\sqrt{i}}{m_i}g_i\left(\frac{m_i}{\sqrt{i}}x\right)$, we can assume that $m_i=\sqrt{i}$ for every $i\in\N$.

Because the sequence $g_i(x)$ is bounded for every $x$, there is a subsequence $g_{i_1}, g_{i_2},\dots$ which is uniformly convergent on every compact subset of $[0,+\infty)$, and has limit $g(x)$. This function is also 1-Lipschitz and has $g(0)=0$, so by definition there exists $t>1$ such that $\Gamma^+_\lambda(g,t)+\Gamma^-_\lambda(g,t)\geq (h(\lambda)-\epsilon/2)t$.

Let $q=\Gamma^+_\lambda(g,t)$. Since $g$ satisfies that $\lambda x+g(x)< t$ for every $0\leq x\leq q$, we also have $\lambda x+g_{i_j}(x)<(1-\delta)t$ for all $x\in [0,q]$, all $\delta>0$ and all $j>J(\delta)$ large enough, therefore $\Gamma^+_\lambda(g_{i_j},(1-\delta)t)>q$. For the same reason, $\Gamma^+_\lambda(g_{i_j},(1-\epsilon/4)t)>\Gamma^-_\lambda(g_{i_j},(1-\delta)t)>\Gamma^-_\lambda(g,t)$. We obtain the inequality\[\frac{\Gamma^+_\lambda(g_{i_j},(1-\delta)t)+\Gamma^-_\lambda(g_{i_j},(1-\delta)t)}{(1-\delta)t}>\frac{\Gamma^+_\lambda(g,t)+\Gamma^-_\lambda(g,t)}{(1-\delta)t}\geq\frac{h(\lambda)-\epsilon/2}{1-\delta}>h(\lambda)-\epsilon\] for $\delta$ small enough. To reach a contradiction, simply observe that $(1-\delta)t\in\left[\frac{1}{\sqrt{i_j}},\sqrt{i_j}\right]=\left[\frac{m_{i_j}}{i_j},m_{i_j}\right]$ for all $j$ large enough.\end{proof}

\section{Properties of $f(\lambda)$}\label{appf}

We will prove four propositions regarding $f(x)$. Propositions \ref{fub} and \ref{flb} together imply \eqref{func}, while Proposition \ref{exactf} means that the upper bound is tight for $x\in[0,1]$.

\begin{prop}\label{continuity} $f$ is non-increasing and continuous.\end{prop}

\begin{proof}
Let $-1<\gamma<\tau<1$. Let $\epsilon>0$. Choose $g$ such that $\limsup\frac{\Gamma^+_\gamma(g,t)+\Gamma^-_\gamma(g,t)}{t}\leq h(\gamma)+\epsilon$. Let $z(t)=\Gamma^+_\gamma(g,t)$ and $\bar z(t)=\Gamma^+_\tau(g,t)$. We have $\gamma z(t)+g(z(t))=t=\tau\bar z(t)+g(\bar z(t))$. Because $\tau z(t)+g(z(t))>\gamma z(t)+g(z(t))=t$, by definition of $\Gamma^+_\tau$ we have $\bar z(t)\leq z(t)$. Thus \[z(t)-\bar z(t)\geq g(z(t))-g(\bar z(t))=\tau\bar z(t)-\gamma z(t)\] which rearranges to $(\tau+1)\bar z(t)\leq (\gamma+1)z(t)$. Similarly, if $z'(t)=\Gamma^-_\gamma(g,t)$ and $\bar z'(t)=\Gamma^-_\tau(g,t)$, then $(\tau+1)z'(t)\leq (\gamma+1)z'(t)$. We take upper limits.

\begin{align*}(\tau+1)h(\tau)\leq& (\tau+1)\limsup\limits_{t\rightarrow\infty}\frac{\bar z(t)+\bar z'(t)}{t}\\ \leq& (\gamma+1)\limsup\limits_{t\rightarrow\infty}\frac{z(t)+z'(t)}{t}\\ \leq& (\gamma+1)h(\gamma)+\epsilon\end{align*} Since this inequality is valid for every $\epsilon>0$, we find that the function $(\gamma+1)h(\gamma)$ is non-increasing on $\gamma$. On the other hand, we have $t=\gamma z(t)+g(z(t))\leq (\gamma+1)z(t)$, or equivalently $z(t)\geq\frac{t}{\gamma+1}$. Similarly $z'(t)\geq \frac{t}{\gamma+1}$. This leads to $h(\gamma)\geq\frac{2}{\gamma+1}$. We conclude that the function $\frac{1-\gamma^2}{2}h(\gamma)+(\gamma+1)$ is non-increasing, because  
\begin{align*}\frac{1-\tau}{2}(\tau+1)h(\tau)+(\tau+1)\leq& \frac{1-\tau}{2}(\gamma+1)h(\gamma)+(\tau+1)\\\leq&\frac{1-\tau}{2}(\gamma+1)h(\gamma)+(\tau+1)-\left(\frac{(\gamma+1)h(\gamma)}{2}-1\right)(\tau-\gamma)\\ =&\frac{1-\gamma}{2}(\gamma+1)h(\gamma)+(\gamma+1).\end{align*} After the change of variables, $\gamma=\frac{\lambda-1}{\lambda+1}$, this is equivalent to $f(\lambda)$ being non-decreasing.

To show that $f$ is continuous, we will show that $h$ is continuous. We will show that, for every $-1<\gamma<1$ and every $\epsilon>0$ there exists $\delta>0$  such that $h(\gamma-\delta)-h(\gamma)>\epsilon$. Because $h$ is non-increasing, this implies continuity.

Choose $\xi>0$ very small, and choose $g$ such that $\frac{\Gamma^+_\gamma(g,t)+\Gamma^-_\gamma(g,t)}{t}\leq h(\gamma)+\xi$ for $t$ large enough.
 Let $z(t)=\Gamma^+_\gamma(g,t)$. We have $\gamma z(t)+g(z(t))=t$. If $t$ is large enough, then $(\gamma-\delta) z(t)+g(z(t))\geq t-\delta z(t)\geq (1-h(\gamma)+\xi)t$. If $\bar z(t)=\Gamma^+_{\gamma-\delta}(g,t)$, then for $t$ large enough we have $\bar z(t)\leq z\left(\frac{t}{1-(h(\gamma)+\xi)\delta}\right)$. As before, the same argument works for $\Gamma^-$, and produces (taking $\xi\rightarrow 0$)\[h(\gamma-\delta)\leq\frac{h(\gamma)}{1-\delta h(\gamma)},\] which is smaller than $h(\gamma)+\epsilon$ for $\delta$ small enough.\end{proof}

\begin{prop}\label{fub}
\begin{equation*}f(\lambda) \leq \left\{
  \begin{array}{ccc}
    \frac{2\lambda^2+3\lambda+7+2\sqrt{\lambda+1}}{4\lambda^2+4\lambda+9} & \text{for} & 0\leq\lambda< 3,\\
    & & \\
    \frac{\lambda+1}{2\lambda} & \text{for} & x \geq 3.
  \end{array}
\right.\end{equation*}
\end{prop}

\begin{proof}

Let $\gamma=\frac{\lambda-1}{\lambda+1}$. All we need to do is find a function $g(x)$ for which \begin{equation*}
\limsup_{t\rightarrow\infty}\frac{\Gamma^+_\gamma(g,t)+\Gamma^-_\gamma(g,t)}{t}\leq\left\{
  \begin{array}{ccc}
    \frac{2\gamma^2+2\gamma+8+\sqrt{32(1-\gamma)}}{(\gamma+1)^3} & \text{for} & \gamma\in\left(-1,\frac12\right)\\
    & & \\
    \frac2\gamma & \text{for} & \gamma \in \left[\frac12,1\right),\end{array}\right.
    \end{equation*}
    
    to show an upper bound on $h(\gamma)$, and consequently on $f(\lambda)$.
    
    For $\gamma\in[\frac12,1)$, take $g(x)=0$. Then $\Gamma^+_\gamma(g,t)=\Gamma^-_\gamma(g,t)=\frac t\gamma$.
    
    For $\gamma\in(-1, \frac12)$, let $\sigma=\frac{1-\gamma+\sqrt{2(1-\gamma)}}{1+\gamma}$. We define $g$ as follows: for every $x>0$, we have $|g(x)|=\min\{|x-\sigma^i|:i\in\mathbb{Z}\}$. $g(x)$ is non-negative if $x\in[\sigma^i,\sigma^{i+1})$ for some odd $i$ and non-positive otherwise. This creates a 1-Lipschitz function in which the derivative is always $\pm1$, and changes sign precisely at the values $x=\frac{\sigma^i+\sigma^{i+1}}{2}$.
    
    This function $g$ satisfies that $\Gamma^+_\gamma(g,t)$ and $\Gamma^-_\gamma(g,t)$ are piecewise linear, with sudden increases at the points at which the respective functions take value $\frac{\sigma^i+\sigma^{i+1}}{2}$. By symmetry, it is at these points that $\frac{\Gamma^+_\gamma+\Gamma^-_\gamma}{2}$ is maximized, and one can check that the value is 
\[\lim\limits_{t\rightarrow (t^*)^+}\frac{\Gamma^+_\gamma(g,t)+\Gamma^-_\gamma(g,t)}{t}=\frac{2(\gamma(\sigma+1)+\sigma^2+2\sigma-1)}{(\gamma+1)((\gamma+1)\sigma+\gamma-1)}=\frac{2\gamma^2+2\gamma+8+\sqrt{32(1-\gamma)}}{(\gamma+1)^3},\] where $t^*=\gamma\frac{\sigma^i+\sigma^{i+1}}{2}+\frac{\sigma^{i+1}-\sigma^i}{2}$.\end{proof}
    
\begin{prop}\label{exactf}
\[f(\lambda)=\frac{2\lambda^2+3\lambda+7+2\sqrt{\lambda+1}}{4\lambda^2+4\lambda+9} \hskip 1cm \forall  0\leq\lambda\leq 1\]
\end{prop}

\begin{proof}

Let $\gamma=\frac{\lambda-1}{\lambda+1}$. If two 1-Lipschitz functions with $g_1(0)=g_2(0)=0$ satisfy $|g_1(x)-g_2(x)|\leq 1$ for every $x$, then $\Gamma^+_\gamma(g_1,t)=\min\{x:\gamma x+g_1(x)\geq t\}\geq\min\{x:\gamma x+g_2(x)\geq t-1\}=\Gamma^+_\gamma(g_2,t-1)$. Similarly, $\Gamma^-_\gamma(g_1,t)\geq \Gamma^-_\gamma(g_2, t-1)$. This implies that \[\limsup_{t\rightarrow\infty}\frac{\Gamma^+_\gamma(g_1,t)+\Gamma^-_\gamma(g_2,t)}{t}=\limsup_{t\rightarrow\infty}\frac{\Gamma^+_\gamma(g_1,t)+\Gamma^-_\gamma(g_2,t)}{t}\] 

For this reason, we can focus our attention on continuous functions $g$ which are piecewise linear, the slope of each piece is 1 or -1 and the length of each piece is at least $\frac12$. Indeed, let $g_1$ be a 1-Lipschitz function with $g_1(0)=0$. Define $g_2$ by starting at $g_2(0)=0$. Take $g_2(x)$ in $[0, x_1]$, where $x_1$ is the minimum value for which $g_1(x_1)=  x_1-1$. Then continue with slope -1 in the interval $[x_1, x_2]$ until the minimum value of $x_2$ for which we would have $g_2(x_2)=g_1(x_2)-1$. Proceed alternating the sign of the slope every time $g_2$ reaches distance 1 from $g_1$. Clearly we always have $|g_1(x)-g_2(x)|\leq 1$, and it is easy to check that each piece has length at least $\frac12$ (it could be that some piece has infinite length).

We can further suppose that the first piece has slope 1. Let $\ell_1, \ell_2, \dots$ be the length of the pieces, and $x_i$ be the end of the $i$-th piece. The points $x_i$ are local maxima if $i$ is odd and minima if $i$ is even. If for some odd $i$ we have $\gamma x_i+g(x_i)\leq \gamma x_{i-2}+g(x_{i-2})$ (which means that $x_i$ does not equal $\Gamma^+_\gamma(g,t)$ for any $t$), we can ``remove the peak'' by extending the intervals $[x_{i-2}, x_{i-1}]$ and $[x_{i+1}, x_{i+2}]$ until they intersect. The new function $\bar g$ satisfies $\Gamma^+_\gamma(\bar g,t)\leq \Gamma^+_\gamma(g,t)$ and $\Gamma^-_\gamma(\bar g,t)=\Gamma^-_\gamma(g,t)$, because $\Gamma^+_\gamma(g,t)$ will always lie on an increasing piece and $\Gamma^-_\gamma(g,t)$ lies on a decreasing piece.

We can similarly ``remove a valley'', if for some even $i$ the value $x_i$ does not equal $\Gamma^-_\gamma(g,t)$ for any $t$. Applying these two procedures repeatedly, always to the first peak or valley that can be removed, we produce a function in which every $x_i$ is either $\Gamma^+_\gamma(g,t_i)$ for odd $i$ or $\Gamma^-_\gamma(g,t_i)$ for even $i$. The function must still have infinitely many peaks and valleys, otherwise we have $\Gamma^+_\gamma(g,t)=\infty$ or $\Gamma^-_\gamma(g,t)=\infty$ for large enough $t$.

Let $\mu>\limsup_{t\rightarrow\infty}\frac{\Gamma^+_\gamma(g,t)+\Gamma^-_\gamma(g,t)}{t}$. After scaling, we can suppose that $\frac{\Gamma^+_\gamma(g,t)+\Gamma^-_\gamma(g,t)}{t}\leq\mu$ for every $t\geq 1$ (because of the scaling, we might lose the property that every interval has length larger than $\frac12$, but it will still be larger than some constant). Let us see the relation that the $t_i$ must satisfy. We have \[x_i=\sum\limits_{j=1}^i\ell_i\hskip 1cm g(x_i)=\sum\limits_{j=1}^i(-1)^{j+1}\ell_i\hskip 1cm t_i=\gamma x_i+(-1)^{i+1}g(x_i).\]

We can express $x_i$ in terms of $t_1, t_2, \dots, t_i$. Combining the identities above we have $t_i=\sum_{j=1}^i(\gamma+(-1)^{j-i})\ell_j$, and therefore $t_i+t_{i-1}=(\gamma+1)\ell_i+2\gamma\sum_{j=1}^{i-1}\ell_j=(\gamma+1)x_i+(\gamma-1)x_{i-1}$, which produces the recursion $x_i=\frac{1-\gamma}{1+\gamma}x_{i-1}+\frac{1}{1+\gamma}(t_i+t_{i-1})$. Together with $x_1=\frac{1}{1+\gamma}t_1$, the solution is \[x_i=\frac{1}{1+\gamma}t_i+\sum\limits_{j=1}^{i-1}\frac{2}{1-\gamma^2}\left(\frac{1-\gamma}{1+\gamma}\right)^{i-j}t_j.\]

Based on these values, we can compute $\Gamma^+_\gamma(g,t)$ and $\Gamma^-_\gamma(g,t)$. The former must lie in the interval $[x_{i-1},x_i]$, where $i$ is the smallest odd value for which $t_i\geq t$. Moreover, it is the unique value $x\in[x_{i-1}, x_i]$ for which $t=\gamma x+g(x)=\gamma x+(g(x_{i})+x-x_{i})=t_i-(\gamma+1)(x_i-x)$. If $m_o=m_o(\{t_i\}_{i=1}^\infty, t)$ is the smallest odd $i$ such that $t_i\geq t$, then \[\Gamma^+_\gamma(g,t)=t_{m_o}-(\gamma+1)(x_{m_0}-x)=\frac{1}{1+\gamma}t+\sum\limits_{j=1}^{m_o-1}\frac{2}{1-\gamma^2}\left(\frac{1-\gamma}{1+\gamma}\right)^{m_o-j}t_j.\]

Similarly, if $m_e=m_e(\{t_i\}_{i=1}^\infty, t)$ is the smallest even $i$ such that $t_i\geq t$, then \[\Gamma^-_\gamma(g,t)=\frac{1}{1+\gamma}t+\sum\limits_{j=1}^{m_e-1}\frac{2}{1-\gamma^2}\left(\frac{1-\gamma}{1+\gamma}\right)^{m_e-j}t_j.\]

Define \[s(\{t_i\}_{i=1}^\infty,t)=\frac{2}{1+\gamma}t+\sum\limits_{j=1}^{m_o-1}\frac{2}{1-\gamma^2}\left(\frac{1-\gamma}{1+\gamma}\right)^{m_o-j}t_j+\sum\limits_{j=1}^{m_e-1}\frac{2}{1-\gamma^2}\left(\frac{1-\gamma}{1+\gamma}\right)^{m_e-j}t_j.\]

We are then interested in the lowest value that $\sup_{t\geq 1}\frac{s(\{t_i\}_{i=1}^{\infty},t)}{t}$ can take. Given a non-negative unbounded sequence $\{t_i\}_{i=1}^\infty$, let $i_1<i_2<\dots$ be the set of indices $i$ for which $t_i\geq 1$ and $t_i>t_j$ for every $j<i$. One can check that $s(\{t_{i_j}\}_{j=1}^\infty,t)\leq s(\{t_i\}_{i=1}^\infty,t)$ for every $t\geq 1$ (this is because $\frac{1-\gamma}{1+\gamma}>1$). Therefore, we can assume that $t_i$ is increasing and $t_1\geq 1$. In this case, and setting $t_0=0$,

\[\sup\limits_{t\geq 1}\frac{ s(\{t_i\}_{i=1}^\infty,t)}{t}\geq\sup\limits_{i\geq 1}\frac{2}{1+\gamma}+\sum\limits_{j=1}^{i+1}\frac{2}{1-\gamma^2}\left(\frac{1-\gamma}{1+\gamma}\right)^{i-j+2}\frac{t_j+t_{j-1}}{t_i},\] since $\{m_o(\{t_i\}_{i=1}^\infty, t_i+\epsilon),m_e(\{t_i\}_{i=1}^\infty, t_i+\epsilon)\}=\{i+1,i+2\}$ for $\epsilon>0$ small enough.

A sequence is called $S$-good if it satisfies\begin{equation}\label{Sgood}St_i\geq\frac{2}{1+\gamma} t_i+\sum\limits_{j=1}^{i+1}\frac{2}{1-\gamma^2}\left(\frac{1-\gamma}{1+\gamma}\right)^{i-j+2}(t_j+t_{j-1})\hskip 1cm \forall i\geq 1,\end{equation} which is a necessary condition for $\sup_{t\geq 1}\frac{s(\{t_i\}_{i=1}^{\infty},t)}{t}\geq S$. Suppose that a sequence is $S$-good. Consider the recurring sequence where $T_1=t_1$ and \begin{equation}\label{recurT}ST_i=\frac{2}{1+\gamma} T_i+\sum\limits_{j=1}^{i+1}\frac{2}{1-\gamma^2}\left(\frac{1-\gamma}{1+\gamma}\right)^{i-j+2}(T_j+T_{j-1}).\end{equation}

Observe that \eqref{recurT} is used to define $T_{i+1}$ from the previous entries, rather than $T_i$. 

\begin{claim} If $\{t_i\}_{i=1}^\infty$ is $S$-good, we have $t_i\leq T_i$ for every $i\geq 1$.\end{claim}
\begin{proof}[Proof of Claim] Suppose that there exists an index $i$ for which $t_i>T_i$. Let $u$ and $v$ be the smallest indices such that $t_u\neq T_u$ and $t_v>T_v$, respectively. Clearly we have $u\geq 2$. We will prove our claim by induction on $v-u$. By definition of $T_i$ we have $u,v>1$. Combining \eqref{Sgood} and \eqref{recurT} for $i=u-1$ we see that we cannot have $v-u=0$. 

Let $\alpha=\frac{T_u-t_u}{t_1}>0$. Consider the sequence $\{t'_i\}_{i=1}^\infty$, where $t'_i=t_i$ for $i<u$ and $t'_i=t_i+\alpha t_{i-u+1}$ for $i\geq u$.  This sequence is still increasing. We have $t'_1=t_1$ and, by the choice of $\alpha$ we have $t_u=T_u$. Additionally, $t'_v\geq t_v>T_v$. This means that, if we define $u'$ and $v'$ analogously to $u$ and $v$, then $u'>u$ and $v'\leq v$, which leads to $v'-u'<v-u$. Finally, observe that $\{t'_i\}$ is still good, since \eqref{Sgood} is satisfied for $i< u$ (the equation becomes \eqref{recurT}) and for $i\geq u$ we have

\begin{align*}St'_i=&St_i+St_{i-u+1}\\=&\frac{2}{1+\gamma} (t_i+t_{i-u+1})+\sum\limits_{j=1}^{i+1}\frac{2}{1-\gamma^2}\left(\frac{1-\gamma}{1+\gamma}\right)^{i-j+2}(t_j+t_{j-1}+t_{j-u+1}+t_{j-u})\\=&\frac{2}{1+\gamma} t'_i+\sum\limits_{j=1}^{i+1}\frac{2}{1-\gamma^2}\left(\frac{1-\gamma}{1+\gamma}\right)^{i-j+2}(t'_j+t'_{j-1})\end{align*}

taking $t_i=0$ for $i\leq 0$. This completes the induction step.\end{proof}

To complete the proof of Theorem \ref{exactf}, we show that $T_i\leq 0$ for some $i$ if $S<\frac{2\gamma^2+2\gamma+8+\sqrt{32(1-\gamma)}}{(\gamma+1)^3}$. The reason is that the sequence $T_i$ also satisfies the recursion \[ST_i-\frac{1-\gamma}{1+\gamma}ST_{i-1}=\frac{2}{1+\gamma}\left(T_i-\frac{1-\gamma}{1+\gamma}T_{i-1}\right)+\frac{2}{1-\gamma^2}\frac{1-\gamma}{1+\gamma}(T_{i+1}-T_i).\]

which can be rewritten as $T_{i+1}+\alpha T_i+\beta T_{i-1}=0$. If $S$ is smaller than the claimed bound, then the roots of the polynomial $x^2+\alpha x+\beta$ are not real. This implies that $T_i\leq 0$ for some $i$ (see for example \cite{BW81}). This proves that $S\geq\frac{2\gamma^2+2\gamma+8+\sqrt{32(1-\gamma)}}{(\gamma+1)^3}$, which produces the desired value of $f(\lambda)$.\end{proof}

\begin{prop}\label{flb} For $x>1$ we have $f(x)\geq \frac{x+1}{2x+1}$.\end{prop}

\begin{proof}We will prove this statement for rational values of $x$, and it will follow for irrational values because $f$ is continuous (Proposition \ref{continuity}).

Let $x=s/r$. Let $F$ be the graph on $r+s$ vertices whose complement is a clique on $r$ vertices (hence $\alpha(F)=r$). We have $\mu(\omega\cdot F, n)=s\lceil n/r\rceil$, because an independent set $I$ intersects at least $\lceil n/r\rceil$ components and has at least $s$ neighbors in each. Combining Theorem \ref{mainub} and Theorem \ref{beslb} we find $f(s/r)\geq \rho(\omega\cdot F)\geq \frac{r+s}{r+2s}$, or $f(x)\geq \frac{x+1}{2x+1}$. \end{proof}

\begin{minipage}[t]{0.5\linewidth}
	Ander Lamaison\\ 
	\texttt{<lamaison@zedat.fu-berlin.de>}\\
	Institut f\"ur Mathematik, Freie Universit\"at Berlin and Berlin Mathematical School, Berlin, Germany. 
\end{minipage}

\end{document}